\documentclass[final]{siamart0216}
\usepackage{amsfonts}
\usepackage{todonotes}
\usepackage{amsfonts,amsmath,amssymb}
\usepackage{mathrsfs,mathtools,stmaryrd,wasysym}
\usepackage{enumerate}
\usepackage{xspace,mydef} 
\usepackage{esint}
\usepackage{graphicx,psfrag}
\usepackage{hyperref}
\usepackage{pgf,tikz,pgfplots}
\usepackage{pstricks-add}

\usetikzlibrary{arrows}
\usepackage{rotating}
\DeclareMathAlphabet{\mathpzc}{OT1}{pzc}{m}{it}

\DeclareMathOperator{\supp}{supp}

\newsiamremark{remark}{Remark}

\newcommand{\dist}{ {\textup{\textsf{d}}}_}
\newcommand{\Sides}{\mathscr{S}}
\newcommand{\Ne}{\mathcal{N}}
\newcommand{\E}{\mathscr{E}}
\newcommand{\Res}{\mathscr{R}}
\newcommand{\Jump}{\mathscr{J}}

\newcommand{\TheTitle}{The stationary Boussinesq problem under singular forcing}
\newcommand{\ShortTitle}{A Boussinesq problem}
\newcommand{\TheAuthors}{A. Allendes, E.~Ot\'arola, A. J.~Salgado}

\headers{\ShortTitle}{\TheAuthors}

\title{{\TheTitle}\thanks{AA has been partially supported by CONICYT through FONDECYT project 1170579. EO has been partially supported by CONICYT through FONDECYT project 11180193. AJS has been partially supported by NSF grant DMS-1720213.}}

\author{
Alejandro Allendes\thanks{Departamento de Matem\'atica, Universidad T\'ecnica Federico Santa Mar\'ia, Valpara\'iso, Chile.
    (\email{alejandro.allendes@usm.cl}, \url{http://aallendes.mat.utfsm.cl/}).}
    \and
  Enrique Ot\'arola\thanks{Departamento de Matem\'atica, Universidad T\'ecnica Federico Santa Mar\'ia, Valpara\'iso, Chile.
    (\email{enrique.otarola@usm.cl}, \url{http://eotarola.mat.utfsm.cl/}).}
  \and
  Abner J.~Salgado\thanks{Department of Mathematics, University of Tennessee, Knoxville, TN 37996, USA.
    (\email{asalgad1@utk.edu}, \url{http://www.math.utk.edu/\string~abnersg})}
}

\ifpdf
\hypersetup{
  pdftitle={\TheTitle},
  pdfauthor={\TheAuthors}
}
\fi

\date{Draft version of \today.}

\begin{document}

\maketitle

\begin{abstract}
In Lipschitz two and three dimensional domains, we study the existence for the so--called Boussinesq model of thermally driven convection under singular forcing. By singular we mean that the heat source is allowed to belong to $H^{-1}(\varpi,\Omega)$, where $\varpi$ is a weight in the Muckenhoupt class $A_2$ that is regular near the boundary. We propose a finite element scheme and, under the assumption that the domain is convex and $\varpi^{-1} \in A_1$, show its convergence. In the case that the thermal diffusion and viscosity are constants, we propose an a posteriori error estimator and show its reliability and local efficiency.
\end{abstract}



\section{Introduction}
\label{sec:intro}

The purpose of this work is to study existence, uniqueness, and approximation results for the so--called Boussinesq model of thermally driven convection. While this problem has been considered before in different contexts and there are such results already available in the literature \cite{MR4053090,MR1364404,MR3454217,farhloul2000mixed,MR1370148,MR3232446,MR2378786,MR2378787}, our main source of novelty and originality here is that we allow the heat source to be singular, say a Dirac measure concentrated in a lower dimensional object so that the problem cannot be understood with the usual energy setting; as it was done, for instance, in \cite{MR1364404,farhloul2000mixed,MR1370148,MR2378786,MR2378787}. Let us make this discussion precise. Let $\Omega \subset \R^d$ with $d \in \{2,3\}$ be an open and bounded domain with Lipschitz boundary $\partial \Omega$. We are interested in existence, uniqueness, and approximation of solutions to the following system of partial differential equations (PDEs) in strong form:
\begin{equation}
\label{eq:BousiStrong}
  \begin{dcases}
    - \DIV( \nu(\Te) \GRAD \ue) + (\ue \cdot \GRAD )\ue + \GRAD \pe = \bg \Te , & \text{ in } \Omega, \\
    \DIV \ue = 0, & \text{ in } \Omega, \\
    -\DIV( \kappa(\Te)\GRAD \Te) + \DIV(\ue \Te) = \calH, &\text{ in } \Omega, \\
    \ue = \boldsymbol0, \ \Te = 0, &\text{ on } \partial\Omega.
  \end{dcases}
\end{equation}
The unknowns are the velocity $\ue$, pressure $\pe$, and temperature $\Te$ of the fluid, respectively. The data are the viscosity coefficient $\nu$, gravity $\bg$, 
the thermal diffusivity coefficient $\kappa$, and the externally applied heat source $\calH$. Our main source of interest is the case of a rough $\calH$, so that standard energy arguments do not apply to obtain suitable estimates. We will make precise assumptions quantifying this below.

Our presentation will be organized as follows. We collect background information, and the main assumptions under which we shall operate in section~\ref{sec:prelim}, where we also introduce a notion of solution for \eqref{eq:BousiStrong}; see Definition~\ref{def:solution}. Existence of solutions is presented in section~\ref{sec:existence}. The numerical analysis of problem \eqref{eq:BousiStrong} begins in section~\ref{sec:Discrete}, where we introduce a finite--element--like numerical scheme, show that it always has solutions, and that these converge. Section~\ref{sec:aposteriori} continues the numerical analysis by introducing an a posteriori error estimator for our problem and showing its reliability and local efficiency. Finally, a series of numerical experiments are presented in section~\ref{sec:numerics}. We show the performance of the devised error estimator within an adaptive loop and explore our model beyond what our theory can handle.

\section{Notation and main assumptions}
\label{sec:prelim}

Throughout this work $d \in \{2,3\}$ and $\Omega \subset \R^d$ is an open and bounded domain with Lipschitz boundary $\partial \Omega$. If $\mathcal{W}$ and $\mathcal{Z}$ are Banach function spaces, we write $\mathcal{W} \hookrightarrow \mathcal{Z}$ to denote that $\mathcal{W}$ is continuously embedded in $\mathcal{Z}$. We denote by $\mathcal{W}'$ and $\|\cdot\|_{\mathcal{W}}$ the dual and the norm of $\mathcal{W}$, respectively.

For $E \subset \Omega$ open and $f : E \to \R$, we set
\[
 \fint_E f \diff x  = \frac{1}{|E|}\int_{E} f \diff x, \qquad |E| = \int_E \diff x.
\]

Given $p \in (1,\infty)$, we denote by $p'$ its H\"older conjugate, \ie the real number such that $1/p + 1/p' = 1$. By $a \lesssim b$ we mean $a \leq C b$, with a constant $C$ that neither depends on $a$, $b$, or the discretization parameters. The value of $C$ might change at each occurrence.

\subsection{Weighted function spaces and their embeddings}
\label{sub:embeddings}

A weight is a locally integrable and nonnegative function defined on $\mathbb{R}^d$. If $\varpi$ is a weight and $p \in [1,\infty)$, we say that $\varpi$ belongs to the so--called Muckenhoupt class $A_p$ if \cite{MR1800316,MR0293384,MR1774162}
\begin{equation*}
\begin{aligned}
\label{A_pclass}
\left[ \varpi \right]_{A_p} & := \sup_{B} \left( \fint_{B} \varpi \diff x\right) \left( \fint_{B} \varpi^{1/(1-p)} \diff x \right)^{p-1}  < \infty, \quad p \in (1,\infty),
\\
\left[ \varpi \right]_{A_1} & := \sup_{B} \left( \fint_{B} \varpi \diff x\right)  \sup_{x \in B} \frac{1}{\varpi(x)}< \infty, \quad p =1,
\end{aligned}
\end{equation*}
where the supremum is taken over all balls $B$ in $\R^d$. In addition, $A_{\infty} := \bigcup_{p \geq 1} A_p$.
We call $[\varpi]_{A_p}$, for $p \in [1,\infty)$, the Muckenhoupt characteristic of $\varpi$.

Distances to lower dimensional objects are prototypical examples of Muckenhoupt weights. In particular, if $\calK \subset \Omega$ is a smooth compact submanifold of dimension $k \in \{ 0, 1, \dots, d-1 \}$ then, owing to \cite[Lemma 2.3 item (vi)]{MR1601373}, we have that 
\begin{equation*}
\label{distance_A2}
  \dist{\calK}^\alpha(x) = \textup{dist}(x,\calK)^{\alpha}
\end{equation*}
belongs to the class $A_p$ provided $\alpha \in \left( -(d-k), (d-k)(p-1) \right)$. This allows us to identify three particular cases:

\begin{enumerate}[(i)]
\item Let $d>1$ and $z \in \Omega$, then the weight $\dist{z}^\alpha \in A_2$ if and only if $ \alpha \in (- d , d) $.
\label{item1}
  
\item Let $d \geq 2$ and $\gamma \subset \Omega$ be a smooth closed curve without self intersections. We have that $\dist{\gamma}^\alpha \in A_2$ if and only if $\alpha \in \left( -(d-1) , d-1 \right) $.
\label{item2}

\item Finally, if $d = 3$ and $\Gamma \subset \Omega$ is a smooth surface without boundary, then $\dist{\Gamma}^\alpha \in A_2$ if and only if $\alpha \in (-1,1)$.
\label{item3}

\end{enumerate}
Since the aforementioned lower dimensional objects are strictly contained in $\Omega$, there is a neighborhood of $\partial\Omega$ where the weight has no degeneracies or singularities. In fact, it is continuous and strictly positive. Inspired by \cite[Definition 2.5]{MR1601373}, this observation motivates us to define a restricted class of Muckenhoupt weights. 

\begin{definition}[class $A_p(D)$]
\label{def:ApOmega}
Let $D \subset \R^d$ be a Lipschitz domain. For $p \in (1, \infty)$ we say that $\varpi \in A_p$ belongs to $A_p(D)$ if there is an open set $\mathcal{G} \subset D$, and positive constants $\varepsilon>0$ and $\varpi_l>0$, such that:
\[
\{ x \in \Omega: \mathrm{dist}(x,\partial D)< \varepsilon\} \subset \mathcal{G}, 
\qquad
\varpi \in C(\bar {\mathcal{G}}), 
\qquad
\varpi_l \leq \varpi(x) \quad \forall x \in \bar{\mathcal{G}}.
\]
\end{definition}

\begin{remark}[$\dist{z}^{\alpha} \in A_2(\Omega)$ and $\dist{z}^{-\alpha} \in A_1$]
\label{rem:weightdelta}
Let $z \in \Omega$ and $\alpha \in (d-2,d)$. Define $\dist{z}(x) = |x-z|$. Then, we have that 
the weight $\dist{z}^{\alpha}$ is such that $\dist{z}^{\alpha} \in A_2(\Omega)$ and $\dist{z}^{-\alpha} \in A_1$.
\end{remark}

From the $A_p$--condition and H\"older's inequality follows that an $A_p$--weight satisfies the so--called \emph{strong doubling property} \cite[Proposition 1.2.7]{MR1774162}: Let $\varpi \in A_p$ with $p \in (1,\infty)$ and $E \subset \mathbb{R}^d$ be a measurable subset of a ball $B \subset \mathbb{R}^d$. Then,
\begin{equation}
\label{eq:strong_double_property}
\varpi(B) \leq C_{p,\varpi} \left( \frac{|B|}{|E|} \right)^p \varpi(E).
\end{equation}

The following embedding results will be of importance in our analysis.

\begin{proposition}[weighted embedding]
\label{prop:WeightedEmb}
Let $p \in (1,\infty)$ and $\varpi \in A_p$. There is $\delta>0$ such that if
\[
  k \in \left[ 1, \frac{d}{d-1} + \delta \right],
\]
then $W^{1,p}_0(\varpi,\Omega) \hookrightarrow L^{kp}(\varpi,\Omega)$. If, in addition, $\mathbf{S}_{\mathrm{sing}}(\varpi) \Subset \Omega$ then, the embedding is compact for $1 \leq k \leq d/(d-1)$. Here, $\mathbf{S}_{\mathrm{sing}}(\varpi)$ denotes the set of singularities defined in \cite[Section 4.1]{MR2797702}.
\end{proposition}
\begin{proof}
Reference \cite[Theorem 1.3]{MR643158} guarantees that the embedding is continuous. The compactness of the embedding follows from \cite[Theorem 4.12]{MR2797702}.
\end{proof}

\begin{proposition}[embedding with different metrics]
\label{prop:embeddingtwometrics}
Let $1 \leq p \leq q < \infty$, $\varpi \in A_p$, and $\rho \in A_q$. If the pair $(\rho,\varpi)$ satisfies the compatibility condition
\[
  \frac{r}{R}\left( \frac{\rho(B_r)}{\rho(B_R)} \right)^{1/q} \left( \frac{\varpi(B_R)}{\varpi(B_r)} \right)^{1/p} \lesssim 1, \qquad 0<r \leq R,
\]
then, we have that, $W^{1,p}_0(\varpi,\Omega) \hookrightarrow L^q(\rho,\Omega)$.
\end{proposition}
\begin{proof}
See \cite[Theorem 6.1]{NOS3}.
\end{proof}

\begin{proposition}[boundedness]
\label{prop:bounded}
Let $d \in \{2,3\}$ and $\varpi \in A_2$. For every $\bg \in \bL^\infty(\Omega)$, $\theta \in H^1_0(\varpi,\Omega)$, and $\bv \in \bH^1_0(\Omega)$, we have that
\begin{equation}
\label{eq:boundedness_1}
  \left| \int_\Omega \theta \bg \cdot \bv \diff x \right| \leq C_{e,1} \| \bg \|_{\bL^\infty(\Omega)} \| \GRAD \theta \|_{\bL^2(\varpi,\Omega)} \| \GRAD \bv \|_{\bL^2(\Omega)}.
\end{equation}
If, in addition, the weight satisfies $\varpi^{-1} \in A_1$, then provided $r \in H^1_0(\varpi^{-1},\Omega)$, we have
\begin{equation}
\label{eq:boundedness_2}
  \left| \int_\Omega \theta \bv \cdot\GRAD r \diff x \right| \leq C_{e,2} \| \GRAD \bv \|_{\bL^2(\Omega)} \| \GRAD \theta \|_{\bL^2(\varpi,\Omega)} \|\GRAD r \|_{\bL^2(\varpi^{-1},\Omega)}.
\end{equation}
In both estimates, the constants depend only on $\Omega$ and $\varpi$.
\end{proposition}
\begin{proof}
Since the weight $\varpi \in A_2$, we have that it satisfies the strong doubling property \eqref{eq:strong_double_property} with $p=2$. Thus, for $0<r \leq R$ and $q \leq d/(d-1)$, we have 
\[
 \frac{r}{R}\left( \frac{ |B_r|}{|B_R|} \right)^{1/q} \left( \frac{\varpi(B_R)}{\varpi(B_r)} \right)^{1/2}
 \lesssim \left( \frac{r}{R} \right)^{1+d/q} \frac{ |B_R|}{|B_r|}
  \lesssim 1.
\]
We thus invoke Proposition~\ref{prop:embeddingtwometrics} with $p=2$ and $\rho=1$ to conclude that $H^1_0(\varpi,\Omega) \hookrightarrow L^q(\Omega)$ provided $q \leq d/(d-1)$. Consequently, for $q \leq d/(d-1)$, we have
\[
  \left| \int_\Omega \theta \bg \cdot \bv \diff x \right| \leq \| \bg \|_{\bL^\infty(\Omega)} \| \theta \|_{L^q(\Omega)} \| \bv \|_{\bL^{q'}(\Omega)} \lesssim \| \bg \|_{\bL^\infty(\Omega)} \| \nabla \theta \|_{L^2(\varpi,\Omega)} \| \bv \|_{\bL^{q'}(\Omega)}.
\]
Notice that $q'\geq d$. Choosing $q=d/(d-1)$ and utilizing a standard Sobolev embedding yield estimate \eqref{eq:boundedness_1}.

We now prove inequality \eqref{eq:boundedness_2}. To accomplish this task, we first notice that $\varpi^{-1} \in A_1$ implies $\varpi \in L^\infty(\Omega)$. Second, since Proposition \ref{prop:WeightedEmb} guarantees the existence of $\epsilon > 0$ such that $H_0^1(\varpi,\Omega) \hookrightarrow L^{\mu+ \epsilon}(\varpi,\Omega)$, for $\mu = 2d/(d-1)$, we can conclude the existence of $\ell \geq 2d/(d-1)$ and $m \leq 2d$ such that $m^{-1} + \ell^{-1} = 1/2$ and
\begin{equation}
\label{eq:estimateBu}
  \begin{aligned}
    \left| \int_\Omega \theta \bv \cdot\GRAD r \diff x \right| 
    & 
    \leq \| \varpi \|_{L^\infty(\Omega)}^{1/m} \| \bv \|_{\bL^m(\Omega)} \| \theta \|_{L^{\ell}(\varpi,\Omega)} \| \GRAD r \|_{\bL^2(\varpi^{-1},\Omega)}
    \\
    &
    \lesssim \| \varpi \|_{L^\infty(\Omega)}^{1/m} \| \nabla \bv \|_{\bL^2(\Omega)} \| \nabla \theta \|_{L^{2}(\varpi,\Omega)} \| \GRAD r \|_{\bL^2(\varpi^{-1},\Omega)},
  \end{aligned}
\end{equation}
where we have also used a standard, unweighted, Sobolev embedding to handle the term involving $\bv$. This yields \eqref{eq:boundedness_2} and concludes the proof.
\end{proof}

\subsection{Main assumptions and definition of solution}
\label{sub:assumptions}

Having described the functional setting that we shall adopt and some of its more relevant properties, we can precisely state the assumptions under which we shall operate.
\begin{enumerate}[$\bullet$]
  \item \emph{Domain:} Let $d \in \{2,3\}$. We assume that $\Omega$ is a bounded domain in $\R^d$ with Lipschitz boundary $\partial\Omega$. When dealing with discretization, we shall further assume that $\Omega$ is a polytope.
  
  \item \emph{Gravity:} The gravity is a constant vector $\bg \in \R^d$. We set $g = |\bg|$.
  
  \item \emph{Viscosity:} The viscosity is a function $\nu \in C^{0,1}(\R)$ that is strictly positive and bounded, \ie there are positive constants $\nu_-$ and $\nu_+$ such that $\nu_- \leq \nu_+$ and
  \[
    \nu_- \leq \nu(t) \leq \nu_+ \quad \forall t \in \R.
  \]
  
  \item \emph{Thermal diffusivity:} The thermal coefficient is a function $\kappa \in C^{0,1}(\R)$ that is, moreover, strictly positive and bounded, \ie there are positive constants $\kappa_-$ and $\kappa_+$ such that $\kappa_- \leq \kappa_+$ and
  \begin{equation*}
  \label{eq:defofkminmax}
    \kappa_- \leq \kappa(t) \leq \kappa_+ \quad \forall t \in \R.
  \end{equation*}
  To quantify the oscillation of the thermal diffusivity we shall introduce
  \begin{equation*}
  \label{eq:defoflambda}
    \Lambda(\kappa) := \frac{\kappa_-}{\kappa_+} \in (0,1].
  \end{equation*}
    
  \item \emph{Weight:} We assume that we have a weight $\varpi \in A_2(\Omega)$ such that $\varpi^{-1} \in A_1$ and $\mathbf{S}_{\mathrm{sing}}(\varpi) \Subset \Omega$.  A canonical example of this scenario is given in Remark~\ref{rem:weightdelta}; \cite[Example 4.4]{MR2797702} shows that $|\mathbf{S}_{\mathrm{sing}}(\varpi)| = 0$.
  
  \item \emph{Heat source:} We allow the heat source to be singular, and we quantify this by assuming that it belongs to the dual of a weighted space. Namely, we assume that $\calH \in H^{-1}(\varpi,\Omega):= H_0^1(\varpi^{-1},\Omega)'$.
\end{enumerate}

With these assumptions at hand we can define our notion of solution.

\begin{definition}[weak solution]
\label{def:solution}
We say that the triple $(\ue,\pe,\Te) \in \bH^1_0(\Omega) \times L^2_0(\Omega) \times H^1_0(\varpi,\Omega)$ is a weak solution to \eqref{eq:BousiStrong} if
\begin{equation}
\label{eq:BusiWeak}
  \begin{dcases}
    \int_\Omega \left( \nu(\Te) \GRAD \ue : \GRAD \bv + (\ue\cdot \GRAD \ue )\cdot \bv - \pe \DIV \bv - \Te\bg \cdot \bv \right) \diff x = 0, 
    \\
    \int_\Omega q \DIV \ue \diff x = 0,
    \\
    \int_\Omega \left( \kappa(\Te) \GRAD \Te \cdot \GRAD r - \Te\ue\cdot\GRAD r \right) \diff x = \langle \calH, r \rangle,
  \end{dcases}
\end{equation}
for all $\bv \in \bH^1_0(\Omega)$, $ q \in L^2_0(\Omega)$, and $r \in H^1_0(\varpi^{-1},\Omega)$. Here, $\langle \cdot, \cdot \rangle$ denotes the duality pairing between $H^1_0(\varpi^{-1},\Omega)$ and its dual $H^{-1}(\varpi,\Omega)$.
\end{definition}

We immediately comment that, owing to our assumptions on data and definition of solution, all terms in this definition are meaningful; see Proposition~\ref{prop:bounded}.

\section{Existence of solutions}
\label{sec:existence}

The main goal in this section is to show that problem \eqref{eq:BousiStrong} has, under the assumptions stated in Section~\ref{sub:assumptions}, a solution in the sense of Definition~\ref{def:solution}. We proceed in several steps.

\subsection{The Navier--Stokes equation with prescribed temperature}
\label{sub:NSE}

We begin by making a simple observation. Given $\theta \in H^1_0(\varpi,\Omega)$, let us consider the following problem: Find $(\bu,p) \in \bH^1_0(\Omega) \times L^2_0(\Omega)$ such that
\begin{equation}
\label{eq:NSEsimp}
  \begin{dcases}
    \int_\Omega \left( \nu(\theta) \GRAD \bu : \GRAD \bv + (\bu\cdot \GRAD \bu )\cdot \bv - p \DIV \bv \right) \diff x = \int_\Omega \theta\bg \cdot \bv  \diff x & \forall \bv \in \bH^1_0(\Omega), 
    \\
    \int_\Omega q \DIV \bu \diff x = 0 & \forall q \in L^2_0(\Omega).
  \end{dcases}
\end{equation}ll the

\begin{theorem}[existence and uniqueness]
\label{thm:existNSEsimp}
For every $\theta \in H^1_0(\varpi,\Omega)$ problem \eqref{eq:NSEsimp} has at least one solution. In addition, if
$
  C_P C_{e,1} g \| \GRAD \theta \|_{\bL^2(\varpi,\Omega)}  < \nu_-^2,
$
where $C_P$ is a constant that depends only on $\Omega$ and $d$, then this solution is unique and satisfies
\[
  \| \GRAD \bu \|_{\bL^2(\Omega)} \leq \frac{C_{e,1}g}{\nu_-} \| \GRAD \theta \|_{\bL^2(\varpi,\Omega)} .
\]
\end{theorem}
\begin{proof}
Since $\theta \in H^1_0(\varpi,\Omega)$, the function $\bar\nu(x) := \nu(\theta(x))$ is bounded, measurable, and strictly positive. Define the functional 
\[
\calF_\theta: \bv \mapsto \int_\Omega \theta \bg \cdot \bv \diff x.
\]
Owing to Proposition~\ref{prop:bounded} we have that $\calF_\theta \in \bH^{-1}(\Omega)$. In addition, Proposition~\ref{prop:bounded} also shows that
\[
  \| \calF_\theta \|_{\bH^{-1}(\Omega)} \leq 
C_{e,1}  \| g\|_{\bL^{\infty}(\Omega)} \| \GRAD \theta \|_{\bL^2(\varpi,\Omega)} 
=
C_{e,1}  g  \| \GRAD \theta \|_{\bL^2(\varpi,\Omega)}.
\]
Thus, the standard theory of existence and uniqueness under small data (or large viscosity) for the Navier--Stokes equation applies \cite[Chapter II, \S 1]{MR1846644}.
\end{proof}

\subsection{The stationary heat equation with convection}
\label{sub:heat}

Here, we study the existence of solutions to a stationary heat equation with convection and under singular forcing. Namely, given $\varkappa \in L^\infty(\Omega)$ with $0<\varkappa_-\leq \varkappa \leq \varkappa_+$, $\bu \in \bH^1_0(\Omega)$ solenoidal, and $\calH \in H^{-1}(\varpi,\Omega)$, we consider the following stationary heat equation: Find $T \in H^1_0(\varpi,\Omega)$ such that
\begin{equation}
\label{eq:SimpHeat}
  \int_\Omega \left( \varkappa \GRAD T \cdot \GRAD r - T\bu\cdot\GRAD r \right) \diff x = \langle \calH, r \rangle \quad \forall r \in H^1_0(\varpi^{-1},\Omega).
\end{equation}

As a first step, we state a well--posedness result for the case $\bu=\boldsymbol0$.

\begin{proposition}[well--posedness for $\bu=\boldsymbol0$]
\label{prop:weightedMeyers}
There is a constant $\Lambda_0$, depending only on $\Omega$ and $\varpi$, such that, if $\Lambda(\varkappa) \geq \Lambda_0$, problem \eqref{eq:SimpHeat} with $\bu = 0$ is well--posed. This, in particular, implies that
\begin{equation}
\label{eq:inf-sup-for-heat}
  \| \GRAD T \|_{\bL^2(\varpi,\Omega)} \leq C_\varkappa \sup_{r \in H^1_0(\varpi^{-1},\Omega)} \frac{ \int_\Omega \varkappa \GRAD T \cdot \GRAD r \diff x}{ \| \GRAD r \|_{\bL^2(\varpi^{-1},\Omega)}} \quad \forall T \in H^1_0(\varpi,\Omega).
\end{equation}
The constant $C_\varkappa$ depends only on $\Lambda_0$, $\Lambda(\varkappa)$, $\Omega$, $d$, and $\varpi$.
\end{proposition}
\begin{proof}
  See \cite[Theorem 12]{OS:17infsup}.
\end{proof}

We now study the case with nonzero convection.

\begin{proposition}[well--posedness for $\bu\neq \boldsymbol0$]
\label{prop:Fredholm}
Assume that $\Lambda(\varkappa) \geq \Lambda_0$, where $\Lambda_0$ is defined in Proposition~\ref{prop:weightedMeyers}. If
\begin{equation}
\label{eq:assumption_grad_bu}
  C_\varkappa C_{e,2}\| \GRAD \bu \|_{\bL^2(\Omega)} \leq q <1,
\end{equation}
then problem \eqref{eq:SimpHeat} is well--posed. This, in particular, implies that the solution $T$ of problem \eqref{eq:SimpHeat} satisfies the estimate
\[
  \| \GRAD T \|_{\bL^2(\varpi,\Omega)} \leq C_H(q) \| \calH \|_{H^{-1}(\varpi,\Omega)}, \qquad C_H(q) = \frac{C_\varkappa}{1-q}.
\]
\end{proposition}
\begin{proof}
Let us introduce the linear map $\calA: H^1_0(\varpi,\Omega) \rightarrow H^{-1}(\varpi,\Omega)$ via
\[
\langle \calA T, r \rangle:=  \int_\Omega\varkappa \GRAD T \cdot \GRAD r  \diff x, \quad \forall T \in H^1_0(\varpi,\Omega), \ \forall r \in H^1_0(\varpi^{-1},\Omega).
\]
Clearly, $\calA$ is a bounded linear operator and, moreover, owing to the inf--sup estimate \eqref{eq:inf-sup-for-heat}, $\calA$ is invertible with $\| \calA^{-1} \|_{\mathcal{L}(H^{-1}(\varpi,\Omega),H^1_0(\varpi,\Omega))} \leq C_\varkappa$.

Given $\bu \in \bH^1_0(\Omega)$, we introduce the map $\calB_\bu : H^1_0(\varpi,\Omega) \to H^{-1}(\varpi,\Omega)$ defined by
\[
  \langle \calB_\bu T, r \rangle = -\int_\Omega T\bu \cdot \GRAD r \diff x, \quad \forall T \in H^1_0(\varpi,\Omega), \forall r \in H^1_0(\varpi^{-1},\Omega).
\]
Estimate \eqref{eq:boundedness_2} shows that $\calB_\bu $ is a bounded linear map which satisfies the estimate
\[
\| \calB_\bu \|_{\mathcal{L}(H^1_0(\varpi,\Omega), H^{-1}(\varpi,\Omega))} \leq C_{e,2} \| \GRAD \bu \|_{\bL^2(\Omega)}.
\]
Since it will be needed later, we now show that $\calB_\bu$ is compact. Let $\{ T_n \}_{n \geq 0}$ be a bounded sequence in $H_0^1(\omega,\Omega)$. Since Proposition~\ref{prop:WeightedEmb} guarantees that, for $k \leq d/(d-1)$, the embedding $H_0^1(\varpi,\Omega) \hookrightarrow L^{2k}(\varpi,\Omega)$ is compact, we conclude the existence of a subsequence $\{ T_{n_j} \}_{j \geq 0}$ of $\{ T_n \}_{n \geq 0}$ such that $ T_{n_j} \to T^*$ in $L^{2k}(\varpi,\Omega)$ as $j \uparrow \infty$. Thus, estimate \eqref{eq:estimateBu} yields
\[
 \| \calB_\bu T_{n_j} -  \calB_\bu T^* \|_{H^{-1}(\varpi,\Omega)} \lesssim  \|  T_{n_j} -   T^* \|_{L^{2k}(\varpi,\Omega)}  \|  \nabla \bu  \|_{L^{2}(\Omega)} \rightarrow 0, \quad j \uparrow \infty.
\]
This shows that $\{  \calB_\bu T_{n_j} \}_{j \geq 0}$ converges in $H^{-1}(\varpi,\Omega)$ and thus that $\calB_\bu$ is compact.
%

With this notation, we have that problem \eqref{eq:SimpHeat} can be written as
\[
  (\calA + \calB_\bu )T  = \calH \qquad \Leftrightarrow \qquad (I+ \calA^{-1} \calB_\bu )T = \calA^{-1} \calH
\]
in $H^{-1}(\varpi,\Omega)$. Since $\calA^{-1} \calB_\bu$ is continuous assumption \eqref{eq:assumption_grad_bu} implies that this problem has a unique solution, because
\[
  \| \calA^{-1} \calB_\bu \|_{\mathcal{L}(H_0^1(\varpi,\Omega))} \leq C_\varkappa C_{e,2} \| \GRAD \bu \|_{\bL^2(\Omega)} \leq q <1.
\]
Moreover, we have the estimate
\begin{align*}
  \| \GRAD T \|_{\bL^2(\varpi,\Omega)} &\leq \frac{ \|\calA^{-1} \|_{\mathcal{L}(H^{-1}(\varpi,\Omega),H^1_0(\varpi,\Omega))}}{ 1 - \| \calA^{-1} \calB_\bu \|_{\mathcal{L}(H_0^1(\varpi,\Omega))}} \| \calH \|_{H^{-1}(\varpi,\Omega)} \leq \frac{C_\varkappa}{1-q} \| \calH \|_{H^{-1}(\varpi,\Omega)}.
\end{align*}
Notice that $C_H(q) = C_\varkappa/(1-q)$ depends only on $q$, $\Lambda_0$, $\Lambda(\varkappa)$, $\Omega$, $d$, and $\varpi$.
\end{proof}

\subsection{Existence of solutions}
\label{sub:existence}

Having studied each one of the subproblems separately, we proceed to show existence of solutions to \eqref{eq:BusiWeak} via a fixed point argument. To accomplish this task, we define the map $\frakF : H^1_0(\varpi,\Omega) \times \bH^1_0(\Omega) \to H^1_0(\varpi,\Omega) \times \bH^1_0(\Omega)$ by $\frakF(\theta,\bu):=(\Te,\ue) $, where $(\Te,\ue)$ solves
\begin{align}
\label{eq:first_eqn_aux}
    & \int_\Omega \left( \nu(\theta) \GRAD \ue : \GRAD \bv + (\ue\cdot \GRAD \ue )\cdot \bv - \pe \DIV \bv \right) \diff x = \int_\Omega \theta\bg \cdot \bv \diff x \quad \forall \bv \in \bH^1_0(\Omega), \\
    \label{eq:second_eqn_aux}
    & \int_\Omega q \DIV \ue \diff x = 0 \quad \forall q \in L^2_0(\Omega), \\
    \label{eq:third_eqn_aux}
    & \int_\Omega \left( \kappa(\theta) \GRAD \Te\cdot\GRAD r  - \Te\ue\cdot\GRAD r \right) \diff x = \langle \calH, r \rangle \quad \forall r \in H^1_0(\varpi^{-1},\Omega).
\end{align}
Note that the definition of $(\Te,\ue)$ implies solving a stationary Navier--Stokes equation with prescribed temperature $\theta$. If this problem has a unique solution, then its velocity component is used to solve a stationary heat equation with convection. The following result shows that the map $\frakF$ is well--defined. To concisely state it we define
\begin{align}
  \frakB_\bu  &= \left\{ \bu \in \bH^1_0(\Omega): \| \GRAD \bu \|_{\bL^2(\Omega)} \leq G \right\}, 
  \quad
  G = \frac1{2C_\varkappa C_{e,2} },
  \label{def:B_u_and_G}
  \\
  \frakB_T &= \left\{ \theta \in H^1_0(\varpi,\Omega): \| \GRAD \theta \|_{\bL^2(\varpi,\Omega)} \leq S \right\}, 
  \quad
  S = \frac{\nu_-}{gC_{e,1}} \min \left\{ \frac{\nu_-}{C_P}, \frac1{2C_\varkappa C_{e,2}} \right\},  
  \label{def:B_T_and_S}
\end{align}
and  $\frakB = \frakB_T \times \frakB_\bu$.

\begin{proposition}[$\frakF$ is well--defined]
\label{prop:Fwd}
Assume that $\Lambda(\kappa) \geq \Lambda_0$, where $\Lambda_0$ is defined in Proposition~\ref{prop:weightedMeyers}. If the heat source $\calH \in H^{-1}(\varpi,\Omega)$ satisfies the estimate
\begin{equation}
  \| \calH \|_{H^{-1}(\varpi,\Omega)} \leq \frac{S}{C_H(1/2)},
 \label{eq:assump_heat_source}
\end{equation}
then $\frakF$ is well--defined on $\frakB$. In addition, we have $\frakF(\frakB) \subset \frakB$.
\end{proposition}
\begin{proof}
Let $\theta \in \frakB_T$. Invoke Theorem~\ref{thm:existNSEsimp} to conclude the existence of a unique $\ue \in \bH^1_0(\Omega)$ that solves \eqref{eq:first_eqn_aux} and \eqref{eq:second_eqn_aux}. Moreover, $\ue$ satisfies the estimate
\[
  \|\GRAD \ue \|_{\bL^2(\Omega)} \leq \frac{C_{e,1}g}{\nu_-} \| \GRAD \theta \|_{\bL^2(\varpi,\Omega)} \leq \frac{C_{e,1}g}{\nu_-} \frac{\nu_-}{gC_{e,1}}\frac1{2C_\varkappa C_{e,2}} =G.
\]
Consequently, $\ue \in \frakB_\bu$. Now, since $\Lambda(\kappa) \geq \Lambda_0$ and $\ue \in \frakB_\bu$, we invoke Proposition~\ref{prop:Fredholm}, with $q=1/2$, to conclude that there exits a unique $\Te$ that solves \eqref{eq:third_eqn_aux}. Moreover, the condition on $\calH$ guarantees that
\[
  \| \GRAD \Te \|_{\bL^2(\varpi,\Omega)} \leq C_{H}(1/2)   \| \calH \|_{H^{-1}(\varpi,\Omega)} \leq S,
\]
which implies that $\Te \in \frakB_T$. We have thus proved the statements of the theorem.
\end{proof}

As a last preparatory step we show that the mapping $\frakF$ is weakly continuous.

\begin{lemma}[weak continuity]
\label{lem:frakFwcont}
The mapping $\frakF: \frakB \to \frakB$ is weakly continuous.
\end{lemma}
\begin{proof}
Let $\{\theta_n \}_{n\geq 0} \subset \frakB_T$ and $\{\bu_n\}_{n\geq 0} \subset \frakB_\bu$ be such that $ (\theta_n,\bu_n) \rightharpoonup (\theta,\bu)$ in $H^1_0(\varpi,\Omega) \times \bH^1_0(\Omega)$. As the set $\frakB = \frakB \times \frakB_{\bu}$ is closed and convex, it is weakly closed. Therefore, $(\theta,\bu) \in \frakB$. Set $(\Te_n,\ue_n) = \frakF(\theta_n,\bu_n)$ and $(\Te,\ue) = \frakF(\theta,\bu)$. We must show that $(\Te_n,\ue_n) \rightharpoonup (\Te,\ue)$.

Owing to the reverse H\"older inequality \cite[Theorem 7.4]{MR1800316} we have that, for some $\epsilon>0$, the embedding $ H^1_0(\varpi,\Omega) \hookrightarrow W^{1,1+\epsilon}(\Omega) $ is continuous. Since $W^{1,1+\epsilon}(\Omega)$ is compactly embedded in $L^{1+\epsilon}(\Omega)$, we obtain that $\theta_n \to \theta $ in $L^{1+\epsilon}(\Omega)$. The continuity of $\kappa$ implies then that $\kappa(\theta_n) \to \kappa(\theta)$ almost everywhere in $\Omega$. Now, since $\{ (\Te_n,\ue_n) \}_{n\geq0} \subset \frakB$ is bounded, we can extract a weakly convergent subsequence $\{(\Te_{n_k},\ue_{n_k})\}_{k \geq 0}$ such that $(\Te_{n_k},\ue_{n_k}) \rightharpoonup (\tilde \Te, \tilde \ue)$ in $H^1_0(\varpi,\Omega) \times \bH^1_0(\Omega)$ as $k \uparrow \infty$. The previous discussion shows that, for every $r \in H^1_0(\varpi^{-1},\Omega)$, we have
\[
 \int_{\Omega} \kappa(\theta_{n_k})\GRAD \Te_{n_k} \cdot \GRAD r \diff x \to  \int_{\Omega} \kappa(\theta)\GRAD \tilde \Te \cdot \GRAD r \diff x, \quad k \uparrow \infty.
\]
Similar arguments for the remaining terms that comprise the definition of $\frakF$ show that, in the limit, we must have $(\tilde \Te, \tilde \ue) = \frakF(\theta, \bu)$. Consequently, $(\tilde \Te, \tilde \ue) = (\Te,\ue)$. Since problem \eqref{eq:first_eqn_aux}--\eqref{eq:third_eqn_aux} admits a unique solution, any convergent subsequence converges to the same limit, which implies that the whole sequence must do so to $(\Te,\ue)$.
\end{proof}

We now proceed to obtain existence via a fixed point argument.

\begin{theorem}[existence]
\label{thm:existenceFull}
Assume that $\Lambda(\kappa) \geq \Lambda_0$, where $\Lambda_0$ is defined in Proposition~\ref{prop:weightedMeyers}. If the heat source $\calH \in H^{-1}(\varpi,\Omega)$ satisfies \eqref{eq:assump_heat_source},
then there is a $(\ue,\pe,\Te) \in \bH^1_0(\Omega) \times L^2_0(\Omega) \times H^1_0(\varpi,\Omega)$ that solves \eqref{eq:BousiStrong} in the sense of Definition~\ref{def:solution}. Moreover, we have that $\ue \in \frakB_\bu $ and $T \in \frakB_T$.
\end{theorem}
\begin{proof}
We wish to invoke the Leray--Schauder fixed point theorem \cite[Theorem 8.8]{MR787404} for the map $\frakF$ over $\frakB = \frakB_T \times \frakB_\bu$, where $\frakB_\bu$ and $\frakB_T$ are defined in \eqref{def:B_u_and_G} and \eqref{def:B_T_and_S}, respectively. Notice that $\frakB$ is nonempty, closed, bounded, and convex. Since Proposition~\ref{prop:Fwd} already showed that $\frakF(\frakB) \subset \frakB$, it remains to show the compactness of $\frakF$. In other words, we must improve on Lemma~\ref{lem:frakFwcont} by showing the weak--strong continuity of $\frakF$. To accomplish this task, let $\{\theta_n\}_{n \geq 0} \subset \frakB_T$ and $\{\bu_n\}_{n \geq 0} \subset \frakB_\bu$ be such that $(\theta_n,\bu_n) \rightharpoonup (\theta,\bu) \in \frakB$, in $H^1_0(\varpi,\Omega) \times \bH^1_0(\Omega)$, as $n \uparrow \infty$. We already now, via Lemma~\ref{lem:frakFwcont}, that $(\Te_n,\ue_n) = \frakF(\theta_n,\bu_n) \rightharpoonup \frakF(\theta,\bu) = (\Te,\ue)$, in $H^1_0(\varpi,\Omega) \times \bH^1_0(\Omega)$.

Let $r \in H^1_0(\varpi^{-1},\Omega)$. Invoke the problems that $(\Te_n,\ue_n)$ and $ (\Te,\ue)$ satisfy and observe that the difference $\ee_{\Te,n} := \Te-\Te_n$ verifies the relation
\[
  \int_\Omega \left( \kappa(\theta_n)\GRAD \ee_{\Te,n} - \ee_{\Te,n} \ue \right)\cdot \GRAD r \diff x = \int_\Omega \left( \Te_n(\ue-\ue_n) + \left( \kappa(\theta_n) - \kappa(\theta)\right)\GRAD \Te  \right) \cdot \GRAD r \diff x,
\]
\ie $\ee_{\Te,n}$ is the solution to a heat equation with convection; the problem that was studied in \S\ref{sub:heat}. Let us denote the functional on the right hand side of this expression by $\calH_n$. Since $\ue \in \frakB_\ue$ and $\Lambda(\kappa) \geq \Lambda_0$, we can invoke Proposition~\ref{prop:Fredholm} to conclude that
\[
  \| \GRAD \ee_{\Te,n} \|_{\bL^2(\varpi,\Omega)} \leq C_H(1/2) \| \calH_n \|_{H^{-1}(\varpi,\Omega)}.
\]
The arguments that led to \eqref{eq:estimateBu} show the existence of $m<2d$ and $\ell > 2d/(d-1)$ such that
\[
  \sup_{r \in H^1_0(\varpi^{-1},\Omega)} \frac{\int_\Omega \Te_n(\ue-\ue_n)\cdot \GRAD r \diff x}{\| \GRAD r \|_{\bL^2(\varpi^{-1},\Omega)}} \lesssim \| \GRAD \Te_n \|_{\bL^2(\varpi,\Omega)} \| \ue - \ue_n \|_{\bL^m(\Omega)} \to 0,
\]
where we have also used the compact embedding $H^1(\Omega) \hookrightarrow L^{m}(\Omega)$. For the second term we observe that, since $\kappa$ is continuous, and hence bounded, we have that $( \kappa(\theta_n) - \kappa(\theta))\GRAD \Te \to \boldsymbol0$ in $\bL^2(\varpi,\Omega)$.

In conclusion, $\Te_n \to \Te$ in $H^1_0(\varpi,\Omega)$. A similar argument shows that $\ue_n \to \ue $ in $\bH^1_0(\Omega)$. The theorem is thus proved.
\end{proof}

\section{Discretization and convergence}
\label{sec:Discrete}

Let us now study a finite--element--like scheme to approximate the solution of \eqref{eq:BousiStrong}. To that effect, we assume that we have at hand, for each $h>0$, finite dimensional spaces $W_h \subset H^1_0(\varpi,\Omega) \cap H^1_0(\varpi^{-1},\Omega)$, $\bX_h \subset \bH^1_0(\Omega)$, and $M_h \subset L^2_0(\Omega)$ that are dense in the limit. Moreover, we assume that the pair $(\bX_h, M_h)$ is compatible, in the sense that there is a constant $\beta>0$ such that, for all $h>0$,
\begin{equation}
\label{eq:LBB}
  \beta \| q_h \|_{L^2(\Omega)} \leq \sup_{\bv_h \in \bX_h} \frac{ \int_\Omega \DIV \bv_h q_h \diff x}{ \| \GRAD \bv_h \|_{\bL^2(\Omega)}} \quad \forall q_h \in M_h.
\end{equation}
We also assume that the $H^1_0(\Omega)$ projection onto $W_h$ is $H^1_0(\varpi^{\pm1},\Omega)$ stable. In other words, there is a constant $\gamma>0$ such that, for all $h>0$,
\begin{equation}
\label{eq:RitzStabWeight}
  \gamma \| \GRAD r_h \|_{\bL^2(\varpi^{\pm1},\Omega)} \leq \sup_{\theta_h \in W_h} \frac{ \int_\Omega \GRAD r_h \cdot \GRAD \theta_h \diff x}{\| \GRAD \theta_h\|_{\bL^2(\varpi^{\mp1},\Omega)}} \quad \forall r_h \in W_h.
\end{equation}
Finally, we assume that there is an interpolation operator $\pi_W : H^1_0(\varpi^{-1},\Omega) \to W_h$ which is stable and has suitable approximation properties: For all $r \in H^1_0(\varpi^{-1},\Omega)$, we have
\begin{equation}
\label{eq:piWstableandApprox}
  \| \GRAD \pi_W r \|_{\bL^2(\varpi^{-1},\Omega)} \lesssim \| \GRAD r \|_{\bL^2(\varpi^{-1},\Omega)}, \quad \| \GRAD (\pi_W r-r) \|_{\bL^2(\varpi^{-1},\Omega)} \overset{h\to0}{\longrightarrow} 0.
\end{equation}

Examples of triples verifying our assumptions are plentiful within the finite element literature. Pairs that satisfy \eqref{eq:LBB} can be found, for instance, in \cite{CiarletBook,MR2050138,MR851383}. In addition, \cite{DO:17} shows that if $\Omega$ is convex, and $W_h$ consists of continuous functions that are piecewise polynomials of degree $k\geq1$ over a quasiuniform mesh of $\Omega$ of size $h$, then \eqref{eq:RitzStabWeight} holds. Finally, in this setting, \cite{NOS3} constructs interpolants that satisfy \eqref{eq:piWstableandApprox}.

As in the continuous case, we will say that a triple $(\ue_h,\pe_h,\Te_h) \in \bX_h \times M_h \times W_h$ is a discrete solution to \eqref{eq:BousiStrong} if
\begin{equation}
\label{eq:Busih}
  \begin{dcases}
    \int_\Omega \left( \nu(\Te_h) \GRAD \ue_h : \GRAD \bv_h + (\ue_h\cdot \GRAD \ue_h )\cdot \bv_h + \frac12 \DIV \ue_h \ue_h \cdot \bv_h \right. \\
    \qquad\qquad \left. - \pe_h \DIV \bv_h - \Te_h\bg \cdot \bv_h \right) \diff x = 0,\\
    \int_\Omega q_h \DIV \ue_h \diff x = 0,\\
    \int_\Omega \left( \kappa(\Te_h) \GRAD \Te_h \cdot \GRAD r_h - \Te_h\ue_h\cdot\GRAD r_h \right) \diff x = \langle \calH, r_h \rangle,
  \end{dcases}
\end{equation}
for all $\bv_h \in \bX_h$, $ q_h \in M_h$, and $r_h \in W_h$. Our main objective here will be to show that, under similar assumptions to Theorem~\ref{thm:existenceFull}, problem \eqref{eq:Busih} always has a solution and that, as $h \to 0$, these solutions weakly converge, up to subsequences, to a solution of \eqref{eq:BousiStrong} in the sense of Definition~\ref{def:solution}.

\subsection{A discrete stationary heat equation with variable coefficient}
\label{sub:discrMeyers}

As a first step to achieve our goals we must prove a discrete version of Proposition~\ref{prop:weightedMeyers}. The proof of the following result is, essentially, an adaption of \cite[Proposition 8.6.2]{MR2373954}.

\begin{proposition}[weighted stability]
\label{prop:discrMeyers}
If $\varkappa \in L^\infty(\Omega)$ is such that $0<\varkappa_-\leq \varkappa \leq \varkappa_+$ and
\begin{equation}
  \Lambda(\varkappa) \geq \Lambda_1 := \max\left\{ \Lambda_0, 1-\gamma\left(1-\frac1{2\varkappa_+} \right)\right\},
  \label{eq:Lambda_1}
\end{equation}
then, for all $h>0$, we have that
\begin{equation}
\label{eq:RitzStabWeight_kappa}
  \frac\gamma2 \| \GRAD r_h \|_{\bL^2(\varpi^{\pm1},\Omega)} \leq \sup_{\theta_h \in W_h} \frac{ \int_\Omega \varkappa \GRAD r_h \cdot \GRAD \theta_h \diff x}{\| \GRAD \theta_h\|_{\bL^2(\varpi^{\mp1},\Omega)}} \quad \forall r_h \in W_h,
\end{equation}
where $\gamma>0$ is the constant appearing in estimate \eqref{eq:RitzStabWeight}.
\end{proposition}
\begin{proof}
As mentioned above, the proof essentially follows the perturbation argument developed in \cite[Proposition 8.6.2]{MR2373954}. Let us define the bilinear form
\[
B:  H^1_0(\varpi^{\pm1},\Omega) \times H^1_0(\varpi^{\mp1},\Omega) \rightarrow \mathbb{R},
 \quad
  B(r,\theta) := \int_\Omega \left( 1 - \frac\varkappa{\varkappa_+} \right)\GRAD r \cdot \GRAD \theta \diff x.
\]
Note that, for $r \in H^1_0(\varpi^{\pm1},\Omega)$ and $\theta \in H^1_0(\varpi^{\mp1},\Omega)$,
\[
  |B(r,\theta)| \leq (1-\Lambda(\varkappa)) \| \GRAD r \|_{\bL^2(\varpi^{\pm1},\Omega)} \| \GRAD \theta \|_{\bL^2(\varpi^{\mp1},\Omega)},
\]
and that
\[
  \int_\Omega  \GRAD r \cdot \GRAD \theta \diff x = B(r,\theta) + \frac1{\varkappa_+} \int_\Omega \varkappa \GRAD r \cdot \GRAD \theta \diff x.
\]
Thus, owing to \eqref{eq:RitzStabWeight} we have that, for any $h>0$ and any $r_h \in W_h$,
\[
  \left(\gamma + \Lambda(\varkappa)-1 \right) \| \GRAD r_h \|_{\bL^2(\varpi^{\pm1},\Omega)} \leq \frac1{\varkappa_+} \sup_{\theta_h \in W_h} \frac{ \int_\Omega \varkappa \GRAD r_h \cdot \GRAD \theta_h \diff x}{\| \GRAD \theta_h\|_{\bL^2(\varpi^{\mp1},\Omega)}}.
\]
The restriction on $\Lambda(\varkappa)$ allows us to conclude.
\end{proof}

With the previous result at hand, we can show that a discrete version of \eqref{eq:SimpHeat} always has a solution for sufficiently small convection and that, more importantly, the discrete solutions are uniformly bounded with respect to $h$. Given $\varkappa \in L^\infty(\Omega)$ with $0<\varkappa_-\leq \varkappa \leq \varkappa_+$, $\bu \in \bH^1_0(\Omega)$, and $\calH \in H^{-1}(\varpi,\Omega)$ we consider the following problem: Find $T_h \in W_h$ such that
\begin{equation}
\label{eq:discrHeat}
  \int_\Omega \left( \varkappa \GRAD T_h \cdot \GRAD r_h - T_h \bu \cdot \GRAD r_h \right) \diff x = \langle \calH, r_h \rangle \quad \forall r_h \in W_h.
\end{equation}

\begin{corollary}[well--posedness]
\label{cor:discrHeat}
Assume that $\Lambda(\varkappa) \geq \Lambda_1$, where $\Lambda_1$ is defined in \eqref{eq:Lambda_1}, and that $\bu \in \bH^1_0(\Omega)$ satisfies
\[
  \frac{2C_{e,2}}\gamma \| \GRAD \bu \|_{\bL^2(\Omega)} \leq q < 1.
\]
Then, for every $h>0$, problem \eqref{eq:discrHeat} has a unique solution. Moreover, $T_h$ satisfies
\[
  \| \GRAD T_h \|_{\bL^2(\varpi,\Omega)} \leq \frac2{\gamma(1-q)} \| \calH \|_{H^{-1}(\varpi,\Omega)}.
\]
\end{corollary}
\begin{proof}
Repeat \emph{verbatim} the proof of Proposition~\ref{prop:Fredholm} replacing \eqref{eq:inf-sup-for-heat} by \eqref{eq:RitzStabWeight_kappa} and $C_\varkappa$ by $2/\gamma$.
\end{proof}

\subsection{Existence and stability}
\label{sub:ExistenceDiscr}

Having studied a discrete diffusion equation with variable coefficient on weighted spaces, we can proceed and show that, under similar assumptions to Theorem~\ref{thm:existenceFull}, our discrete problem \eqref{eq:Busih}, always has solutions and that, moreover, these are uniformly bounded with respect to $h>0$. This will be the first step to show, via a compactness argument, the convergence of discrete solutions to a solution of \eqref{eq:BousiStrong}, in the sense of Definition~\ref{def:solution}. 

We proceed via a fixed point argument. We define, for each $h>0$, the map 
\[
\frakF_h : W_h \times \bX_h \to W_h \times \bX_h, \quad (\theta_h,\bu_h) \mapsto \frakF_h(\theta_h,\bu_h) = (\Te_h,\ue_h)
\]
by the following procedure: Let the pair $(\ue_h,\pe_h) \in \bX_h \times M_h$ be a solution to
\begin{equation}
\label{eq:defofDiscrFrakF1}
  \begin{dcases}
    \int_\Omega \left( \nu(\theta_h) \GRAD \ue_h : \GRAD \bv_h + (\ue_h\cdot\GRAD) \ue_h \cdot \bv_h + \frac12 \DIV \ue_h \ue_h \cdot \bv_h \right) \diff x  \\
      \qquad \qquad -\int_\Omega \pe_h \DIV \bv_h \diff x = \int_\Omega \theta_h\bg \cdot \bv_h \diff x & \forall \bv_h \in \bX_h, \\
    \int_\Omega q_h \DIV \ue_h \diff x = 0 & \forall q_h \in M_h,
  \end{dcases}
\end{equation}
then, $\Te_h \in W_h$ is found as the solution of
\begin{equation}
\label{eq:defofDiscrFrakF2}
\int_\Omega \left( \kappa(\theta_h) \GRAD \Te_h \cdot \GRAD r_h - \Te_h \ue_h \cdot \GRAD r_h \right) \diff x = \langle \calH, r_h \rangle \quad \forall r_h \in W_h.
\end{equation}

For reasons similar to the continuous case, this map is well--defined, if we restrict it to a ball of appropriate size. To quantify that, we introduce
\begin{align*}
  \frakB_\bu^h  &= \left\{ \bu_h \in \bX_h: \| \GRAD \bu_h \|_{\bL^2(\Omega)} \leq \tilde G \right\}, 
  \quad
  \tilde G = \frac\gamma{4 C_{e,2} }, 
  \\
  \frakB_T^h &= \left\{ \theta_h \in W_h: \| \GRAD \theta_h \|_{\bL^2(\varpi,\Omega)} \leq \tilde S \right\}, 
  \quad
 \tilde S = \frac{\nu_-}{gC_{e,1}} \min \left\{ \frac{\nu_-}{C_P}, \frac\gamma{4 C_{e,2}} \right\},
\end{align*}
and $\frakB^h = \frakB_T^h \times \frakB_\bu^h$. Notice that neither $\tilde S$ nor $\tilde G$ are dependent on the parameter $h>0$.

\begin{proposition}[$\frakF_h$ is well--defined]
\label{prop:frakFhmakessense}
Assume that $\Lambda(\kappa) \geq \Lambda_1$, where $\Lambda_1$ is defined in Proposition~\ref{prop:discrMeyers}. If the heat source $\calH \in H^{-1}(\varpi,\Omega)$ satisfies the estimate
\[
  \| \calH \|_{H^{-1}(\varpi,\Omega)} \leq \frac{\gamma\tilde S}4,
\]
then, for every $h>0$, the mapping $\frakF_h$ is well--defined on $\frakB^h$. Moreover, $\frakF_h(\frakB^h) \subset \frakB^h$.
\end{proposition}
\begin{proof}
The proof essentially repeats that of Proposition~\ref{prop:Fwd}. For this reason, we only sketch it. Let $(\theta_h,\bu_h) \in \frakB^h = \frakB_T^h \times \frakB_\bu^h$. Since we have skew symmetrized the convective term, we know that problem \eqref{eq:defofDiscrFrakF1} always has a solution which satisfies
\[
  \| \GRAD \ue_h \|_{\bL^2(\Omega)} \leq \frac1{\nu_-} \| \theta_h \bg \|_{\bH^{-1}(\Omega)} \leq \frac{C_{e,1}g}{\nu_-} \| \GRAD \theta_h \|_{\bL^2(\varpi,\Omega)} \leq \tilde G,
\]
where we used \eqref{eq:boundedness_1} and the fact that $\theta_h \in \frakB_T^h$. In addition, we observe that the conditions on the data for this solution to be unique are met \cite[Theorem IV.3.1]{MR548867}. This, now unique, $\ue_h \in \bX_h$ can be used as datum in \eqref{eq:defofDiscrFrakF2}. Corollary~\ref{cor:discrHeat}, with $q=1/2$, then implies that this problem has a unique solution, which satisfies
\[
  \| \GRAD \Te_h \|_{\bL^2(\varpi)} \leq \frac4\gamma \| \calH \|_{H^{-1}(\varpi)} \leq \tilde S,
\]
where we used the assumption on $\calH$. Thus, as we intended to show, $\frakF_h$ is well--defined on $\frakB^h$ and $\frakF_h(\frakB^h) \subset \frakB^h$.
\end{proof}

We conclude by showing existence of solutions, via a fixed point argument. Since we are in finite dimensions this is much easier now.

\begin{theorem}[existence]
\label{thm:ExistenceDiscr}
Assume that $\Lambda(\kappa) \geq \Lambda_1$, where $\Lambda_1$ is defined in Proposition~\ref{prop:discrMeyers}. If the heat source $\calH \in H^{-1}(\varpi,\Omega)$ satisfies the estimate
\[
  \| \calH \|_{H^{-1}(\varpi,\Omega)} \leq \frac{\gamma\tilde S}4,
\]
then, for every $h>0$, there is a triple $(\ue_h,\pe_h,\Te_h) \in \bX_h \times M_h \times W_h$ that solves \eqref{eq:Busih}. Moreover, we have that $\ue_h \in \frakB^h_\bu$ and $\Te_h \in \frakB_T^h$.
\end{theorem}
\begin{proof}
Since we are now in finite dimensions, we will apply Brouwer's fixed point theorem \cite[Theorem 3.2]{MR787404}. For that, we only need to verify the continuity of $\frakF_h$. This is achieved by repeating \emph{verbatim} the proof of Lemma~\ref{lem:frakFwcont} and using that we are in finite dimensions to pass from weak to strong convergence.
\end{proof}

\subsection{Convergence}
\label{sub:convergence}

The results of the previous section show that, provided the heat source $\calH$ and the oscillation of $\kappa$ are not too large,
then for every $h>0$ problem \eqref{eq:Busih} has a solution, and that this family of solutions remains bounded uniformly in $h>0$. We can then pass to a (not relabeled) weakly convergent subsequence $(\Te_h,\ue_h) \rightharpoonup (\Te,\ue)$. We will show here that this limit must be a solution to \eqref{eq:BousiStrong}, in the sense of Definition~\ref{def:solution}.

We begin with some notation. We define
\begin{align*} 
  \hat \frakB_\bu  &= \left\{ \bu \in \bH^1_0(\Omega) : \| \GRAD \bu \|_{\bL^2(\Omega)} \leq \hat G \right\},
  &\hat\frakB_\bu^h &= \hat\frakB_\bu \cap \bX_h,
  &\hat G &= \min\left\{ G, \tilde G \right\},
  \\
  \hat\frakB_T &= \left\{ \theta \in H^1_0(\varpi,\Omega): \| \GRAD \theta \|_{\bL^2(\varpi,\Omega)} \leq \hat S \right\}, 
  &\hat\frakB_T^h &= \hat\frakB_T \cap W_h, &  \hat S &= \min \left\{ S, \tilde S \right\},
\end{align*}
$\hat \frakB =  \hat\frakB_T \times \hat \frakB_\bu$, and $H = \min\left\{ \frac{\gamma\tilde S}4, \frac{2S}{C_\kappa} \right\}$.

\begin{theorem}[convergence]
\label{thm:bigConv}
Assume that $\Lambda(\kappa) \geq \Lambda_1$, where $\Lambda_1$ is defined in Proposition~\ref{prop:discrMeyers}. If the heat source $\calH \in H^{-1}(\varpi,\Omega)$ satisfies
\[
  \| \calH \|_{H^{-1}(\varpi)} \leq H,
\]
then the family $\{ (\Te_h,\ue_h) \}_{h>0} \subset \frakB_T^h \times \frakB_\bu^h$ of solutions to \eqref{eq:Busih} converges weakly (up to subsequences) to an element of $\frakB_T \times \frakB_\bu$. Moreover, this limit is a solution to \eqref{eq:BousiStrong} in the sense of Definition~\ref{def:solution}.
\end{theorem}
\begin{proof}
The assumptions guarantee that we can invoke Theorem~\ref{thm:ExistenceDiscr} to ascertain the existence of $(\Te_h,\ue_h) \in \hat \frakB_T^h \times \hat \frakB_\bu^h$ that solve \eqref{eq:Busih}. Moreover, since $\hat \frakB_T^h \times \hat \frakB_\bu^h \subset \hat \frakB$, this family of solutions remains in a bounded set, and we can extract weakly convergent subsequences which for simplicity of notation we do not relabel. Let us denote this limit by $(\Te,\ue) \in \hat \frakB$, and show that it is a solution to \eqref{eq:BusiWeak}.

Observe now that:
\begin{enumerate}[$\bullet$]
  \item The compact embedding $H^1_0(\varpi,\Omega) \hookrightarrow \hookrightarrow L^{1+\epsilon}(\Omega)$ implies, by continuity of $\nu$ and $\kappa$, that we have $\nu(\Te_h) \to \nu(\Te)$ and $\kappa(\Te_h) \to \kappa(\Te)$ almost everywhere in $\Omega$.
  
  \item Let $q \in C_0^\infty(\Omega)$ with zero average, and $q_h \in M_h$ its $L^2$--projection onto $M_h$. Then we have that
  \[
    \left| \int_\Omega q \DIV \ue \diff x \right| \leq \left| \int_\Omega (q-q_h) \DIV \ue_h \diff x \right| + \left| \int_\Omega q \DIV(\ue - \ue_h) \diff x \right| \to 0,
  \]
  where we used that $\ue_h$ is discretely solenoidal, the strong convergence $q_h \to q$ in $L^2_0(\Omega)$ and the weak convergence $\DIV \ue_h \rightharpoonup \DIV \ue$. In conclusion $\ue$ is solenoidal.
  
  \item We now show that the pair $(\Te,\ue)$ satisfies the heat equation with convection. Let now $r \in H^1_0(\varpi^{-1},\Omega)$ be arbitrary and introduce $r_h = \pi_W r \in W_h$, where the operator $\pi_W$ is the one that satisfies \eqref{eq:piWstableandApprox}. The almost everywhere convergence of $\kappa(\Te_h)$ then implies
  \begin{align*}
    \int_\Omega \kappa(\Te_h) \GRAD \Te_h \cdot \GRAD r_h \diff x &= \int_\Omega \kappa(\Te_h) \GRAD \Te_h \cdot \GRAD r \diff x + \int_\Omega \kappa(\Te_h) \GRAD \Te_h \cdot \GRAD(r_h- r) \diff x \\
      &\to \int_\Omega \kappa(\Te) \GRAD \Te \cdot \GRAD r \diff x.
  \end{align*}
Finally, let $k > 2d/(d-1)$ such that we have that $\Te_h \to \Te$ strongly in $L^k(\varpi, \Omega)$. Since we also have the compact embedding $H^1(\Omega) \hookrightarrow \hookrightarrow L^m(\Omega)$, with $k^{-1} + m^{-1} = 1/2$, \ie $m < 2d$, we know that $\ue_h \rightarrow \ue$ in $\bL^m(\Omega)$. This is sufficient to assert that
  \[
    \int_\Omega \Te_h \ue_h \cdot \GRAD r \diff x \to \int_\Omega \Te \ue \cdot \GRAD r \diff x.
  \]
  
  \item It remains to deal with the momentum equation, but most of the terms are standard here and have been treated in several other works, see for instance \cite{MR3543019,MR3579306,MR3125906,MR2970745,MR3127973}. The only somewhat nonstandard term is
  \[
    \int_\Omega \Te_h \bg \cdot \bv_h \diff x,
  \]
  but the estimates of Proposition~\ref{prop:bounded} can be used to assert convergence of this term as well.
\end{enumerate}
In conclusion, the limit is a solution and the theorem is proved.
\end{proof}

\section{A posteriori error estimates}
\label{sec:aposteriori}

In this section, we design and analyze an a posterior error estimator for the finite dimensional approximation \eqref{eq:Busih} of problem \eqref{eq:BusiWeak}. To be able to do so, in addition to the assumptions stated in section~\ref{sub:assumptions}, we shall require that:
\begin{enumerate}[$\bullet$]
  \item The viscosity $\nu$ and the thermal diffusivity $\kappa$ are independent of the temperature, \ie they are \emph{positive constants}.
\end{enumerate}

In addition, to be able to develop an explicit a posteriori error estimator and show its reliability and efficiency, we must be more specific in the structure of the discrete spaces we are dealing with. For this reason, in addition to the assumptions of section~\ref{sec:Discrete}, we assume that, for $h>0$, the spaces $W_h$, $\bX_h$, and $M_h$ are constructed using finite elements over a conforming and shape regular mesh $\T_h = \{K\}$ of $\bar \Omega$. In this setting, however, the parameter $h$ does not bear  the meaning of a mesh size. Rather, it can be thought of as $h=1/k$, where $k \in \mathbb{N}_0$ is the index set in a sequence of refinements of an original partition $\T_0$. For definiteness, we select the pair of discrete velocity/pressure spaces $(\bX_h,M_h)$ from the following (popular) options:
\begin{enumerate}[1.]
\item
The lowest order Taylor Hood element \cite{hood1974navier}, \cite{MR993474}, \cite[Section 4.2.5]{MR2050138}, which is defined by
\begin{align}
\label{eq:th_V}
\mathbf{X}_h & = \left\{  \bv_h \in \mathbf{C}(\bar \Omega): \ \forall K \in \T_h, \bv_h|_{K} \in \mathbb{P}_2(K)^d \right\} \cap \bH_0^1(\Omega),
\\
\label{eq:th_P}
M_h & = \left\{ q_h \in L^2(\Omega)/\R \cap C(\bar \Omega):\  \forall K \in \T_h, q_h|_{K} \in \mathbb{P}_1(K) \right\}.
\end{align}

\item The mini element, which is considered in \cite{MR799997}, \cite[Section 4.2.4]{MR2050138} and is defined by
\begin{align}
\label{eq:mini_V}
\mathbf{X}_h & = \left\{  \bv_h \in \mathbf{C}(\bar \Omega):\ \forall K \in \T_h, \bv_h|_{K} \in [\mathbb{P}_1(K) \oplus \mathbb{B}(K)]^d \right\} \cap \bH_0^1(\Omega),
\\
\label{eq:mini_P}
M_h & = \left\{  q_h \in  L^2(\Omega)/\R \cap C(\bar \Omega):\ \forall K \in \T_h, q_h|_{K} \in \mathbb{P}_1(K) \right\},
\end{align}
where $\mathbb{B}(K)$ denotes the space spanned by local bubble functions.
\end{enumerate}
Both pairs satisfy the compatibility condition \eqref{eq:LBB} and are such that $\mathbf{X}_h \subset \mathbf{H}_0^1(\Omega)$ and $M_h \subset L_0^2(\Omega)$. We will set the finite element space $W_h$ as
\begin{equation*}
W_{h} := \left\{ w_h \in C(\bar \Omega): \ \forall K \in \T_h, w_h|_{K} \in \mathbb{P}_k(K) \right\} \cap H_0^1(\Omega),
\end{equation*}
where $k=2$ when the Taylor--Hood element \eqref{eq:th_V}--\eqref{eq:th_P} is used to approximate the velocity and pressure variables and $k=1$ for when the mini element \eqref{eq:mini_V}--\eqref{eq:mini_P} is considered. Notice that, for any $\varpi \in A_2$, $W_h \subset W^{1,\infty}_0(\Omega) \subset H^1_0(\varpi^{\pm1},\Omega)$.

We begin our analysis by introducing some preliminary notions. We define the velocity error $\ee_\ue$, the pressure error $\ee_{\pe}$, and the temperature error $\ee_{\Te}$ as follows:
\begin{equation*}
\label{eq:errors_definition}
\ee_\ue:= \ue - \ue_h \in  \bH^1_0(\Omega), 
\quad
\ee_{\pe}:= \pe - \pe_h \in L^2_0(\Omega),
\quad
\ee_{\Te} = \Te - \Te_h \in H^1_0(\varpi,\Omega).
\end{equation*}
We also define, for an open set $D \subset \Omega$, the following norms on the space $\bH^1_0(D) \times L^2(D) \times H^1_0(\varpi,D)$:
\begin{equation*}
\begin{aligned}
\label{eq:energy_norm}
\interleave  (\bw, s, w ) \interleave_{D}^2 & := \nu \| \nabla \bw \|_{\bL^2(D)}^2 + \| \DIV \bw \|_{L^2(D)}^2 + \| s \|_{L^2(D)}^2  + \| \nabla w \|_{\bL^2(\varpi,D)}^2,
\\
 \| (\bw, s, w ) \|_{D}^2 & := \nu \| \nabla \bw \|_{\bL^2(D)}^2 + \| s \|_{L^2(D)}^2  + \| \nabla w \|_{\bL^2(\varpi,D)}^2.
\end{aligned}
\end{equation*}

\subsection{Ritz projection}
To perform a reliability analysis for the devised a posteriori error estimator we shall introduce a so--called Ritz projection $(\boldsymbol{\Phi}, \psi, \varphi)$ of the residuals. This projection is defined as the solution to the following problem: Find $(\boldsymbol{\Phi}, \psi, \varphi) \in \bH^1_0(\Omega) \times L^2_0(\Omega) \times H^1_0(\varpi,\Omega)$ such that
\begin{align}
\label{eq:RP_1}
 \nu \int_{\Omega} \nabla \boldsymbol{\Phi} : \nabla \bv \diff x & = \Xi(\bv), \qquad \forall \bv \in \bH^1_0(\Omega),
\\
\label{eq:RP_2}
\int_{\Omega} \psi q \diff x & =  \Sigma(q), \qquad \forall q \in L_0^2(\Omega),
\\
\label{eq:RP_3}
\int_{\Omega} \nabla \varphi \cdot \nabla r \diff x & =  \Upsilon(r), \qquad \forall r \in  H^1_0(\varpi^{-1},\Omega),
\end{align}
where the functionals $\Xi \in \bH^{-1}(\Omega)$, $\Sigma\in L_0^2(\Omega)$, and $\Upsilon \in H^{-1}(\varpi,\Omega)$ are defined, respectively, by
\begin{align}
\Xi( \bv) & := \int_{\Omega} \left( \nu \nabla \ee_\ue: \nabla \bv -  \ee_{\pe} \DIV \bv  \right. \nonumber \\&\left. +(\ue\cdot\GRAD)\ee_\ue\cdot \bv + (\ee_\ue \cdot \GRAD)\ue_h \cdot \bv + \frac12 \DIV \ee_\ue \ue_h \cdot \bv
- \ee_{\Te}  \mathbf{g} \cdot \bv \right) \diff x,
\label{eq:Xi}
\\
\nonumber
\Sigma(q) &:= - \int_{\Omega} q \DIV \ee_{\ue} \diff x,
\\
\nonumber
\Upsilon(r) & := \int_{\Omega} \left( \kappa \nabla \ee_{\Te} \cdot \nabla r -  \Te \ee_{\ue} \cdot \nabla r -  \ee_{\Te} \ue_h \cdot \nabla r \right)\diff x.
\end{align}

The following result yields the well--posedness of problem \eqref{eq:RP_1}--\eqref{eq:RP_3}. Here, and in what follows, by $C_{4\to2}$ we denote the best constant in the Sobolev embedding $\bH^1_0(\Omega) \hookrightarrow \bL^4(\Omega)$. To shorten notation we define
\begin{equation*}
\mathfrak{B}(\ue,\ue_h):= \frac{C_{4\to2}^2}\nu\left( \| \nabla \ue \|_{\bL^2(\Omega)} + \frac32 \| \nabla \ue_h \|_{\bL^2(\Omega)}\right),
\quad
\mathfrak{A}(\ue,\ue_h):= 1 + \mathfrak{B}(\ue,\ue_h).
\end{equation*}

\begin{proposition}[Ritz projection]
Problem \eqref{eq:RP_1}--\eqref{eq:RP_3} admits a unique solution $( \boldsymbol{\Phi}, \psi,  \varphi) \in \bH^1_0(\Omega) \times L^2_0(\Omega) \times H^1_0(\varpi,\Omega)$. In addition, we have the estimate
\begin{multline}
\label{eq:estimate_Ritz}
\| ( \boldsymbol{\Phi}, \psi,  \varphi) \|_{\Omega}^2 \leq 
 \left[ 3 \mathfrak{A}(\ue,\ue_h)^2
 + 2\frac{C_{1}^2C_{e,2}^2}\nu \|\GRAD \Te \|_{\bL^2(\varpi,\Omega)}^2 \right] \nu \| \GRAD \ee_\ue \|_{\bL^2(\Omega)}^2 
 \\ + \| \DIV \ee_\ue \|_{L^2(\Omega)}^2 + \frac3\nu \| \ee_\pe \|_{L^2(\Omega)}^2 \\
 + \left[\frac{3g^2C_{e,1}^2}{\nu} + 2 C_1^2\left( \kappa + C_{e,2}\| \GRAD \ue_h \|_{\bL^2(\Omega)} \right)^2 \right] \| \GRAD \ee_{\Te} \|_{\bL^2(\varpi,\Omega)}^2,
\end{multline}
where $C_1$ denotes the constant in the inf-sup estimate \eqref{eq:inf-sup-for-heat} with $\varkappa = 1$.
\end{proposition}
\begin{proof}
Since $\Xi \in \bH^{-1}(\Omega)$, the Lax--Milgram Lemma immediately yields the existence of a unique $\boldsymbol{\Phi} \in \bH^1_0(\Omega)$ that solves problem \eqref{eq:RP_1}.  In addition, estimate \eqref{eq:boundedness_1} and the standard Sobolev embedding $\bH^1_0(\Omega) \hookrightarrow \bL^4(\Omega)$ yield
\begin{multline}
\label{eq:bound_Phi}
  \sqrt{\nu} \| \nabla \boldsymbol{\Phi} \|_{\bL^2(\Omega)} \leq
  \mathfrak{A}(\ue,\ue_h)
  \sqrt{\nu} \| \GRAD \ee_\ue \|_{\bL^2(\Omega)} \\ + \frac1{\sqrt{\nu}} \| \ee_\pe\|_{L^2(\Omega)} + \frac{gC_{e,1}}{\sqrt{\nu}} \| \GRAD \ee_{\Te}\|_{\bL^2(\varpi,\Omega)}.
\end{multline}
On the other hand, since $\ee_{\ue} \in \bH^1_0(\Omega)$, similar arguments reveal the existence and uniqueness of $\psi \in L^{2}_0(\Omega)$ that solves problem \eqref{eq:RP_2} together with the bound
\begin{equation} 
\label{eq:bound_psi}
\| \psi \|_{L^2(\Omega)} \leq \| \DIV \ee_{\ue} \|_{L^2(\Omega)}.
\end{equation}
Next, we invoke the inf--sup condition \eqref{eq:inf-sup-for-heat} for the variational form of the Dirichlet Laplace operator on weighted spaces to conclude that there exists a unique $\varphi \in H^1_0(\varpi,\Omega)$ that solves \eqref{eq:RP_3}. In addition, we have that $\varphi$ satisfies the estimate
\begin{multline}
\label{eq:bound_varphi}
\| \nabla \varphi \|_{\bL^2(\varpi,\Omega)} \leq C_1 \left[ \kappa \| \nabla \ee_{\Te} \|_{\bL^2(\varpi,\Omega)} +  C_{e,2} \| \nabla \Te \|_{\bL^2(\varpi,\Omega)}  \| \nabla \ee_{\ue} \|_{\bL^2(\Omega)}  \right.
\\
\left. +  C_{e,2} \| \nabla \ee_{\Te} \|_{\bL^2(\varpi,\Omega)}  \| \nabla \ue_h \|_{\bL^2(\Omega)} \right],
\end{multline}
where $C_1$ denotes the constant in the inf-sup estimate \eqref{eq:inf-sup-for-heat} with $\varkappa = 1$.

The desired estimate \eqref{eq:estimate_Ritz} thus follows from collecting estimates \eqref{eq:bound_Phi}, \eqref{eq:bound_psi}, and \eqref{eq:bound_varphi}. This concludes the proof.
\end{proof} 

\subsection{An upper bound for the error}
We now prove that the energy norm of the error can be bounded in terms of the energy norm of the Ritz projection, which in turn will allow us to provide a computable upper bound for the error.

\begin{theorem}[upper bound for the error]
\label{thm:upper_bound_for_error}
Assume that the solutions to \eqref{eq:BusiWeak} and \eqref{eq:Busih} are such that the following inequalities hold:
\begin{equation}
\begin{aligned}
  \max\left\{ C_\kappa C_{e,2}, \frac{3C_{4\to2}^2}{2\nu} \right\} \| \GRAD \ue_h \|_{\bL^2(\Omega)} &\leq \frac1{16} \\
  \frac{C_{4\to2}^2}{\nu} \| \GRAD \ue \|_{\bL^2(\Omega)} &\leq \frac1{16} \\
  \frac{g C_\kappa C_{e,1}C_{e,2}}\nu  \| \GRAD \Te \|_{\bL^2(\varpi,\Omega)} &\leq \frac{15}{16^2}.
\end{aligned}
\label{eq:assumptions_for_Ritz}
\end{equation}
Then, we have that
\begin{equation*}
\interleave (\ee_{\ue}, \ee_{\pe}, \ee_{\Te} ) \interleave_{\Omega} \lesssim \interleave  (\boldsymbol{\Phi},\psi , \varphi ) \interleave_{\Omega},
\end{equation*}
where the hidden constant is independent of $(\ee_{\ue}, \ee_{\pe}, \ee_{\Te} )$ and $ (\boldsymbol{\Phi},\psi , \varphi )$ but depends on $\nu$ and the constant involved in the inf--sup condition \eqref{eq:inf_sup}.
\end{theorem}
\begin{proof}
We divide the proof in six steps.

\emph{Step 1.} We first bound $ \| \nabla \ee_{\Te}\|_{\bL^2(\varpi,\Omega)}$. Owing to Proposition \ref{prop:weightedMeyers} we have that there is a positive constant $C_{\kappa}$ such that the following inf--sup condition holds:
\begin{equation}
\label{eq:inf_sup_e_T}
 \| \nabla \ee_{\Te}\|_{\bL^2(\varpi,\Omega)} \leq C_{\kappa} \sup_{r \in H_0^1(\varpi^{-1},\Omega)} \frac{\int_{\Omega} \kappa \nabla \ee_{\Te} \cdot \nabla r \diff x}{\| \nabla r\|_{\bL^2(\varpi^{-1},\Omega)}}.
\end{equation}
To estimate the right hand side of \eqref{eq:inf_sup_e_T} we rewrite equation \eqref{eq:RP_3} as 
\[
\int_{\Omega} \kappa \nabla \ee_{\Te} \cdot \nabla r \diff x =  \int_{\Omega} \left( \Te \ee_{\ue} \cdot \nabla r  + \ee_{\Te} \ue_h \cdot \nabla r  +  \nabla \varphi \cdot \nabla r \right) \diff x \quad \forall r \in H_0^1(\varpi^{-1},\Omega).
\]
Utilize the inf--sup condition \eqref{eq:inf_sup_e_T} and estimate \eqref{eq:boundedness_2}, twice, to obtain
\begin{multline*}
\| \nabla \ee_{\Te}\|_{\bL^2(\varpi,\Omega)} \leq C_{\kappa} \left( C_{e,2}\| \nabla \Te \|_{\bL^2(\varpi,\Omega)}  \| \nabla \ee_{\ue} \|_{\bL^2(\Omega)} \right.
\\
\left. + C_{e,2} \| \nabla \ee_{\Te} \|_{\bL^2(\varpi,\Omega)}  \| \nabla \ue_h \|_{\bL^2(\Omega)} + \| \nabla \varphi \|_{\bL^2(\varpi,\Omega)} \right),
\end{multline*}
which, upon utilizing the first estimate in \eqref{eq:assumptions_for_Ritz}, yields the bound
\begin{equation}
\label{eq:estimate_for_ue_T}
\frac{15}{16}
\| \nabla \ee_{\Te}\|_{\bL^2(\varpi,\Omega)} \leq C_{\kappa} \left( C_{e,2}  \| \nabla \Te \|_{\bL^2(\varpi,\Omega)}  \| \nabla \ee_{\ue} \|_{\bL^2(\Omega)} +  \| \nabla \varphi \|_{\bL^2(\varpi,\Omega)} \right).
\end{equation}

\emph{Step 2.} We now control  $\| \ee_{\pe} \|_{L^2(\Omega)}$. To do this, we rewrite equation \eqref{eq:RP_1} as follows:
\begin{multline*}
\int_{\Omega} \ee_{\pe} \DIV \bv \diff x = \int_{\Omega} \left( \nu \nabla \ee_\ue: \nabla \bv   +  (\ue \cdot \nabla) \ee_\ue \cdot  \bv + (\ee_\ue \cdot \nabla) \ue_h \cdot  \bv \right.
\\
\left. + \frac12 \DIV \ee_\ue \ue_h \cdot \bv - \ee_{\Te}  \mathbf{g} \cdot \bv - \nu\nabla \boldsymbol{\Phi} : \nabla \bv\right) \diff x =: \mathfrak{X}(\bv) \quad \forall \bv \in \bH_0^1(\Omega).
\end{multline*}
A standard inf--sup condition for the divergence thus yields the estimate
\begin{equation}
\beta \| \ee_{\pe} \|_{L^2(\Omega)} \leq  \| \mathfrak{X} \|_{\bH^{-1}(\Omega)}.
\label{eq:inf_sup}
\end{equation}
Moreover, Sobolev embeddings and estimate \eqref{eq:boundedness_1} imply that
\begin{multline*}
\|\mathfrak{X}\|_{\bH^{-1}(\Omega)} \leq \nu \left( \| \nabla \ee_\ue \|_{\bL^2(\Omega)} + \| \nabla \boldsymbol{\Phi} \|_{\bL^2(\Omega)} \right) \\
  + \mathfrak{B}(\ue,\ue_h)\nu \| \nabla \ee_\ue \|_{\bL^2(\Omega)}+  C_{e,1}g \| \nabla \ee_{\Te} \|_{\bL^2(\varpi,\Omega)} .
\end{multline*}
To conclude this step, we invoke \eqref{eq:estimate_for_ue_T} to arrive at
\begin{multline}
\label{eq:estimate_for_ee_pe}
  \beta \| \ee_{\pe} \|_{L^2(\Omega)} \leq   \nu ( \| \nabla \boldsymbol{\Phi} \|_{\bL^2(\Omega)} + \| \nabla \ee_\ue \|_{\bL^2(\Omega)} ) + \frac{16C_\kappa C_{e,1}g}{15} \| \nabla \varphi \|_{\bL^2(\varpi,\Omega)} \\
+
\left[ 
   \mathfrak{B}(\ue,\ue_h)
 + \frac{16 C_\kappa C_{e,1} C_{e,2}g}{15\nu} \| \nabla \Te \|_{\bL^2(\varpi,\Omega)}
\right]  \nu\| \nabla \ee_{\ue} \|_{\bL^2(\Omega)} .
\end{multline}

\emph{Step 3.} Set $\bv = \ee_{\ue}$ in \eqref{eq:RP_1} and $q = -\ee_{\pe}$ in \eqref{eq:RP_2}. Adding the obtained relations, and using the skew symmetry of convection when the first argument is solenoidal, we see that
\begin{multline*}
\nu \| \nabla \ee_{\ue} \|^2_{\bL^2(\Omega)} \leq \nu \| \nabla \boldsymbol{\Phi} \|_{\bL^2(\Omega)} \| \nabla \ee_{\ue} \|_{\bL^2(\Omega)} + \| \psi \|_{L^2(\Omega)}\| \ee_{\pe}  \|_{L^2(\Omega)} \\
  + \frac{3 C_{4\to2}^2}{2\nu} \| \GRAD \ue_h \|_{\bL^2(\Omega)}\nu \|\GRAD \ee_\ue \|_{\bL^2(\Omega)}^2 + C_{e,1}g \|\GRAD \ee_{\Te} \|_{\bL^2(\varpi,\Omega)} \| \nabla \ee_{\ue} \|_{\bL^2(\Omega)}.
\end{multline*}
Estimate \eqref{eq:estimate_for_ue_T} and Young's inequality imply that
\begin{multline*}
  \nu (1-\varepsilon_1) \| \nabla \ee_{\ue} \|^2_{\bL^2(\Omega)} \leq \nu \| \nabla \boldsymbol{\Phi} \|_{\bL^2(\Omega)} \| \nabla \ee_{\ue} \|_{\bL^2(\Omega)} + \| \ee_\pe \|_{L^2(\Omega)} \| \psi \|_{L^2(\Omega)} \\
    + \left[\frac{3 C_{4\to2}^2}{2\nu}\| \GRAD \ue_h \|_{\bL^2(\Omega)} + \frac{16 C_\kappa C_{e,1}C_{e,2}g}{15\nu}\|\GRAD \Te \|_{\bL^2(\varpi,\Omega)} \right]\nu \|\GRAD \ee_\ue \|_{\bL^2(\Omega)}^2 \\
    + C_{\varepsilon_1} \left(\frac{16 C_\kappa C_{e,1} C_{e,2} g}{15 {\sqrt{\nu}} } \right)^2 \| \GRAD \varphi \|_{\bL^2(\varpi,\Omega)}^2.
\end{multline*}
The pressure error estimate \eqref{eq:estimate_for_ee_pe}, and repeated applications of Young's inequality imply that, for some $\varepsilon \in (0,1)$
\begin{multline*}
  (1-\varepsilon)\nu \| \GRAD \ee_\ue \|_{\bL^2(\Omega)}^2 \leq C_{\varepsilon,\beta}\nu\| \GRAD \boldsymbol{\Phi} \|_{\bL^2(\Omega)}^2  + C_{\varepsilon,\beta,{\nu}} \| \psi \|_{L^2(\Omega)}^2 
 + C_{\varepsilon,\beta,{\nu}} \| \GRAD \varphi \|_{\bL^2(\varpi,\Omega)}^2  \\+
 \left[\frac{3 C_{4\to2}^2}{2\nu}\| \GRAD \ue_h \|_{\bL^2(\Omega)} + \frac{16 C_\kappa C_{e,1}C_{e,2}g}{15\nu}\|\GRAD \Te \|_{\bL^2(\varpi,\Omega)} \right]\nu \|\GRAD \ee_\ue \|_{\bL^2(\Omega)}^2 \\
 + \frac12
\left[ \mathfrak{B}(\ue,\ue_h) + \frac{16 C_\kappa C_{e,1} C_{e,2}g}{15\nu} \| \nabla \Te \|_{\bL^2(\varpi,\Omega)} \right]^2 \nu\| \nabla \ee_{\ue} \|_{\bL^2(\Omega)}^2 .
\end{multline*}
Now, estimates \eqref{eq:assumptions_for_Ritz} imply that we can choose $\varepsilon \in (0,1)$ so that
\[
  \nu \| \nabla \ee_{\ue} \|^2_{\bL^2(\Omega)}  \lesssim \nu \| \nabla \boldsymbol{\Phi} \|_{\bL^2(\Omega)}^2 +  \| \psi \|^2_{L^2(\Omega)} 
+ \| \nabla \varphi \|^2_{\bL^2(\varpi,\Omega)},
\]
where the hidden constant depends on $\epsilon$, $\beta$, and $\nu$.

\emph{Step 5}. The previous estimate combined with estimates \eqref{eq:estimate_for_ue_T} and \eqref{eq:estimate_for_ee_pe} yield
\[
  \nu \| \nabla \ee_{\ue} \|^2_{\bL^2(\Omega)} + \| \ee_\pe \|_{L^2(\Omega)}^2 + \| \GRAD \ee_{\Te} \|_{\bL^2(\varpi,\Omega)}^2 \lesssim \nu \| \nabla \boldsymbol{\Phi} \|_{\bL^2(\Omega)}^2 +  \| \psi \|^2_{L^2(\Omega)} 
+ \| \nabla \varphi \|^2_{\bL^2(\varpi,\Omega)}.
\]

\emph{Step 6}. Conclude with the obvious observation that, since $\DIV \ee_\ue \in L^2_0(\Omega)$, \eqref{eq:RP_2} implies
\[
  \| \DIV \ee_\ue \|_{L^2(\Omega)} \leq \| \psi \|_{L^2(\Omega)}.
\]
 
The claimed upper bound for the error has been obtained, and the theorem has been proved.
\end{proof}

\subsection{A residual--type error estimator}
In this section, we design an a posteriori error estimator for the finite dimensional approximation \eqref{eq:Busih} of problem \eqref{eq:BusiWeak}. To be to be able to do so, we will assume that the singular forcing $\mathcal{H}$ has a particular structure, that is:
\begin{enumerate}[$\bullet$]
  \item The singular forcing term $\calH$ has the form $\mathcal{H} = \mathfrak{h} \delta_{z}$, where $\delta_{z}$ corresponds to the Dirac delta supported at the interior point $z \in \Omega$ and $\mathfrak{h} \in \mathbb{R}$.
\end{enumerate}
To handle such a singular forcing term, we introduce the weight $\dist{z}^{\alpha}$, where $\dist{z}(x) := |x-z|$ and $\alpha \in (d-2,d)$. We must immediately notice the following two important properties: First, owing to Remark \ref{rem:weightdelta}, we have that the weight $\dist{z}^{\alpha}$ is such that $\dist{z}^{\alpha} \in A_2(\Omega)$ and $\dist{z}^{-\alpha} \in A_1$. Second, $\delta_{z} \in H^{-1}(\dist{z}^{\alpha},\Omega)$;  see \cite[Lemma 7.1.3]{KMR} and \cite[Remark 21.18]{MR2305115}. Simply put, all the assumptions we have made so far apply to this particular choice of $\calH$ and weight.

\subsubsection{Notation}

Before presenting and analyzing our a posteriori error estimator we first need to introduce and set some notation. We recall that $\T_h = \{K\}$ is a conforming and shape regular partition of $\bar\Omega$ into closed simplices $K$ with size $h_K = \diam(K)\approx |K|^{1/d}$. We denote by $\Sides$ the set of \emph{internal} $(d-1)$--dimensional interelement boundaries $S$ of $\T_h$. For $S \in \Sides$, we indicate by $h_S$ the diameter of $S$. If $K\in\T_h$, we define $\Sides_{K}$ as the subset of $\Sides$ that contains the sides of $K$. For $S \in \Sides$, we set $\Ne_S = \{ K^+, K^-\}$, where $K^+, K^- \in \T$ are such that $S=K^+ \cap K^-$. For $K \in \T_h$, we define the following \emph{stars} or \emph{patches} associated with the element $K$
\begin{equation}
\Ne_K := \left\{ K' \in \T_h : \Sides_K \cap \Sides_{K'} \neq \emptyset \right\},
\qquad
\mathcal{S} _K := \{ K' \in \T_h : K \cap K' \neq \emptyset \}.
\label{eq:NeTstar}
\end{equation}
In an abuse of notation, below we denote by $\Ne_K$ and $\mathcal{S}_K$ either the sets themselves, or the union of its elements.

\subsubsection{A posteriori error estimator}\label{sub_section_estimator}
We define an error estimator that can be decomposed as the sum of two contributions: a contribution related to the discretization of the stationary Navier--Stokes equations and another one associated to the discretization of the stationary heat equation with convection and singular forcing.

To present the contribution related to the stationary Navier--Stokes equations, we define, for an element $K \in \T_h$ and an internal side $S \in \Sides$, the \emph{element residual} $\Res_K$ and the \emph{interelement residual} $\! \Jump_S$ as
\begin{equation}\label{eq:residuals_navier}
\begin{aligned}
\Res_K & := \left(\nu \Delta \ue_h - ( \ue_h \cdot \nabla ) \ue_h - \tfrac{1}{2} \DIV \ue_h \ue_h -  \nabla \pe_h + \Te_h \bg\right)|_K,
\\
\Jump_S &:= \llbracket( \nu \nabla \ue_h - \pe_h \mathbf{I} ) \cdot \boldsymbol{n} \rrbracket,
\end{aligned}
\end{equation}
where $(\ue_{h},\pe_h,\Te_h)$ denotes a solution to the discrete problem \eqref{eq:Busih} and  $\mathbf{I} \in \R^{d \times d}$ denotes the identity matrix. For a discrete tensor valued function $\bv_h$, we denote by $\llbracket \bv_h \cdot \boldsymbol{n} \rrbracket$ the jump, which is defined, on the internal side $S \in \Sides$ shared by the distinct elements $K^+$, $K^{-} \in \mathcal{N}_S$, by
\begin{equation*}
\label{eq:jump}
 \llbracket \bv_h \cdot \boldsymbol{n} \rrbracket =  \bv_h|_{K^+}\cdot \boldsymbol{\nu}^+ +  \bv_h|_{K^-} \cdot \boldsymbol{n}^-.
\end{equation*}
Here $\boldsymbol{n}^+, \boldsymbol{n}^-$ are unit normals on $S$ pointing towards $K^+$, $K^{-}$, respectively. With $\Res_K$ and $\! \Jump_S$ at hand, we define, for $K \in \T_h$, the \emph{element indicator}
\begin{equation}
\label{eq:estimator_Navier}
\E_K^2:= h_K^2 \| \Res_K \|^2_{\bL^2(K)} 
+ \| \DIV  \ue_h  \|^2_{L^2(K)} + h_K \|  \Jump_S  \|^2_{\bL^2(\partial K \setminus \partial \Omega)}.
\end{equation}

We now introduce the contribution associated to the stationary heat equation with convection. To accomplish this task, we define, for $K \in \T_h$ and an internal side $S \in \Sides$, the \emph
{element residual} $\mathfrak{R}_K$ and the \emph{interelement residual} $\mathfrak{J}_S$ as
\begin{equation}
\begin{aligned}
\mathfrak{R}_K & := \left(\kappa \Delta \Te_h  - \ue_h \cdot \nabla \Te_h - \DIV \ue_h \Te_h \right)|_K,
\\
\mathfrak{J}_S &:= \llbracket (\kappa \nabla \Te_h - \Te_h \ue_h) \cdot \boldsymbol{n} \rrbracket.
\end{aligned}
\label{eq:residuals_heat}
\end{equation}
With $\mathfrak{R}_K$ and $\mathfrak{J}_S$ at hand, we define, for $K \in \T_h$ and $\alpha \in (d-2,d)$, 
\begin{equation}
\label{eq:estimator_heat}
\mathfrak{E}_K^2:= h_K^2 D_K^{\alpha} \| \mathfrak{R}_K  \|^2_{L^2(K)} 
+ h_K D_K^{\alpha} \| \mathfrak{J}_S \|^2_{L^2(\partial K \setminus \partial \Omega)} + |\mathfrak{h}|h_K^{\alpha+2-d} \# (\{z\} \cap K),
\end{equation}
where for a set $E$, by $\#(E)$ we mean its cardinality. Thus $\#(\{z\} \cap K)$ equals one if $z \in K$ and zero otherwise. Here we must recall that we consider our elements $K$ to be closed sets.

With all these ingredients at hand, we define the local error indicator $\mathcal{E}_K^2:= \E_K^2 + \mathfrak{E}_K^2$ and the a posteriori error estimators
\begin{equation}
\E_h := \left[ \sum_{K \in \T_h} \E_K^2\right]^{\frac{1}{2}},
\quad
\mathfrak{E}_h := \left[ \sum_{K \in \T_h} \mathfrak{E}_K^2\right]^{\frac{1}{2}},
\quad
\mathcal{E}_h^2:= \left[ \sum_{K \in \T_h} \mathcal{E}_K^2 \right]^{\frac{1}{2}}.
\end{equation}

\subsection{Reliability estimates}

We present the following global reliability estimate for the estimator $\mathcal{E}_h$.

\begin{theorem}[global reliability]
Let $(\ue,\pe, \Te) \in \bH_0^1(\Omega) \times L_0^2(\Omega) \times H_0^1(\dist{z}^{\alpha},\Omega)$ be a solution to \eqref{eq:BusiWeak} and $(\ue_{h},\pe_{h}, \Te_h) \in \bX_h \times M_h \times W_h$ be its finite element approximation obtained as solution to \eqref{eq:Busih}. Let $\alpha \in (d-2,d)$. In the framework of Theorem \ref{thm:upper_bound_for_error}, we have the following a posteriori error estimate:
\begin{equation}
\interleave(\ee_{\ue}, \ee_{\pe}, \ee_{\Te} ) \interleave_{\Omega} \lesssim \mathcal{E}_h,
\label{eq:reliability}
\end{equation}
where the hidden constant is independent of the continuous and discrete solutions, the size of the elements in the mesh $\T_h$, and $\#\T_h$.
\end{theorem}
\begin{proof}
We proceed in several steps.

\emph{Step 1.} Invoke Theorem \ref{thm:upper_bound_for_error}, and the obvious bound $\| \DIV \boldsymbol{\Phi} \|_{L^2(\Omega)} \leq \| \nabla \boldsymbol{\Phi} \|_{\bL^2(\Omega)}$ to arrive at the estimate
\[
\interleave(\ee_{\ue}, \ee_{\pe}, \ee_{\Te} ) \interleave_{\Omega}^2 \lesssim \| \nabla \boldsymbol{\Phi} \|_{\bL^2(\Omega)}^2 + \| \psi \|_{L^2(\Omega)}^2  + \| \nabla \varphi \|_{\bL^2(\dist{z}^{\alpha},\Omega)}^2.
\]
It thus suffices to bound $\| \nabla \boldsymbol{\Phi} \|_{\bL^2(\Omega)}$, $\| \psi \|_{L^2(\Omega)}$, and $\| \nabla \varphi \|_{\bL^2(\dist{z}^{\alpha},\Omega)}$. 

\emph{Step 2.} We control $\| \nabla \boldsymbol{\Phi} \|_{\bL^2(\Omega)}$. To accomplish this task, we invoke equation \eqref{eq:RP_1}, the fact that $(\ue,\pe,\Te)$ solves problem \eqref{eq:BusiWeak}, and an integration by parts formula to conclude that, for every $\bv \in \bH_0^1(\Omega)$, we have
\begin{equation}
\mathrm{I}:= \int_{\Omega} \nu \nabla \boldsymbol{\Phi} : \nabla \bv \diff x  = \sum_{K \in \T_h} \int_{K} \Res_{K} \cdot \bv \diff x + \sum_{S \in \Sides} \int_S \Jump_S \cdot \bv \diff s.
\label{eq:term_I}
\end{equation}
Denote by $I_h$ the Cl\'ement interpolation operator. We utilize the first equation of problem \eqref{eq:Busih} with $\bv_h = I_h \bv$ and an integration by parts formula, again, to arrive at 
\begin{equation*}
\mathrm{I} = \sum_{K \in \T_h} \int_{K} \Res_{K} \cdot (\bv - I_h \bv)  \diff x + \sum_{S \in \Sides} \int_S \Jump_S \cdot (\bv - I_h \bv)  \diff s.
\end{equation*}
We now invoke standard approximation properties for the interpolation operator $I_h$ and a scaled trace inequality to conclude that
\begin{equation*}
|\mathrm{I}|
\lesssim \sum_{K \in \T_h} \left(h_K \|  \Res_{K}\|_{\bL^2(K)} \| \nabla \bv \|_{\bL^2(\mathcal{S}_K)} 
+
\sum_{S \in \Sides_K} h_K^{\frac{1}{2}} \| \Jump_S \|_{\bL^2(S)}  \| \nabla \bv \|_{\bL^2(\mathcal{S}_K)}\right).
\end{equation*}
Set $\bv = \boldsymbol{\Phi}$, use the Cauchy--Schwarz inequality in $\mathbb{R}^{\# \T_h}$ and the finite overlapping property of stars to obtain
\begin{equation}
\| \nabla \boldsymbol{\Phi} \|_{\bL^2(\Omega)} \lesssim \left(\sum_{K \in \T_h}  h_K^2\|  \Res_{K}\|^2_{\bL^2(K)}   + h_K  \| \Jump_S \|^2_{\bL^2(\partial K \setminus \partial \Omega)} \right)^{\frac{1}{2}} .
 \label{eq:est_Phi}
\end{equation}

\emph{Step 3.} In this step we bound $ \|\psi \|_{L^2(\Omega)}$. Set $q = \psi$ in \eqref{eq:RP_2}. This yields
\begin{equation}
  \begin{aligned}
    \|\psi \|_{L^2(\Omega)}^2 &=  \int_{\Omega} \psi \DIV \ue_h \diff x \leq \sum_{K \in \T_h} \|\psi \|_{L^2(K)} \| \DIV \ue_h \|_{L^2(K)} \\
    &\leq \left( \sum_{K \in \T_h} \| \DIV \ue_h \|_{L^2(K)} \right)^{\frac12} \| \psi \|_{L^2(\Omega)}.
  \end{aligned}
\label{eq:est_psi}
\end{equation}
Notice that we have used that $\int_{\Omega} q \DIV \ue \diff x = 0$ for every $q \in L_0^2(\Omega)$.

\emph{Step 4.} Estimates \eqref{eq:est_Phi} and \eqref{eq:est_psi} immediately yield that
\begin{equation}
\label{eq:est_Phi_and_psi}
  \left( \| \nabla \boldsymbol{\Phi} \|_{\bL^2(\Omega)}^2 + \|\psi \|_{L^2(\Omega)}^2 \right)^{\frac12} \lesssim 
  \left(\sum_{K \in \T_h} \E_K^2 \right)^{\frac{1}{2}} = \E_h.
\end{equation}

\emph{Step 5.} Our goal now is to bound $\| \nabla \varphi \|_{\bL^2(\dist{z}^{\alpha},\Omega)}$. To accomplish this task, we invoke the problem that $\varphi$ solves, \ie problem \eqref{eq:RP_3}, and the fact that $(\ue,\pe,\Te)$ solves \eqref{eq:BusiWeak}. These arguments, combined with an integration by parts formula, yield
\begin{equation*}
\mathrm{II}:= \int_{\Omega} \nabla \varphi \cdot \nabla r \diff x  = \langle \mathfrak{h} \delta_z , r \rangle +
 \sum_{K \in \T_h} \int_{K} \mathfrak{R}_{K} r \diff x + \sum_{S \in \Sides} \int_S \mathfrak{J}_S r \diff s,
\end{equation*}
for every $r \in H_0^1(\dist{z}^{-\alpha},\Omega)$. We recall that $\mathfrak{R}_{K}$ and $\mathfrak{J}_S$ are defined in \eqref{eq:residuals_heat}. Invoke the discrete problem \eqref{eq:Busih} and an integration by parts formula, again, to arrive at
\begin{equation*}
\mathrm{II} = \langle \mathfrak{h} \delta_z , r - \pi_W r \rangle +
 \sum_{K \in \T_h} \int_{K} \mathfrak{R}_{K} (r-\pi_W r) \diff x + \sum_{S \in \Sides} \int_S \mathfrak{J}_S (r-\pi_W r) \diff s,
\end{equation*}
where we recall that $\pi_W$ denotes the quasi--interpolation operator onto $W_h$, constructed in \cite{NOS3}, that satisfies \eqref{eq:piWstableandApprox}. We control $\langle \mathfrak{h} \delta_z , r - \pi_W r \rangle$ on the basis of \cite[Theorem 4.7]{AGM} and stability and interpolation estimates for $\pi_W$ derived in \cite{NOS3}. In fact, let $K$ be such that $z \in K$, then
\begin{equation*}
\begin{aligned}
|\langle \mathfrak{h} \delta_z , r - \pi_W r \rangle| & \lesssim |\mathfrak{h}|  h_K^{\frac{\alpha}{2}-\frac{d}{2}} \| r - \pi_W r \|_{L^2(\dist{z}^{-\alpha},K)} + |\mathfrak{h}| h_K^{\frac{\alpha}{2}+ 1-\frac{d}{2}} \| \nabla( r - \pi_W r) \|_{\bL^2(\dist{z}^{-\alpha},K)}
\\
& \lesssim 
|\mathfrak{h}| h_K^{\frac{\alpha}{2}+ 1-\frac{d}{2}} \| \nabla r \|_{\bL^2(\dist{z}^{-\alpha},\mathcal{S}_K)}.
\end{aligned}
\end{equation*}
Notice now that
\begin{align*}
 \int_{K} \mathfrak{R}_{K} (r-\pi_W r) \diff x &\leq \| \mathfrak{R}_{K}\|_{L^2(K)} \| r-\pi_W r \|_{L^2(K)} \\ &\lesssim h_K D_{K}^{\frac{\alpha}{2}}\| \mathfrak{R}_{K}\|_{L^2(K)} \| \nabla r \|_{\bL^2(\dist{z}^{-\alpha},\mathcal{S}_K)},
\end{align*}
where we have used \cite[Proposition 4]{MR3892359}. Similar arguments yield, upon using \cite[Proposition 5]{MR3892359}, the control of the jump term. We can thus invoke the inf--sup condition \eqref{eq:inf-sup-for-heat} to arrive at the estimate
 \begin{equation}
 \| \nabla \varphi \|_{\bL^2(\dist{z}^{\alpha},\Omega)} \lesssim \sup_{r \in H_0^1(\dist{z}^{-\alpha},\Omega)} \frac{\int_{\Omega} \nabla \varphi \cdot \nabla r}{\| \nabla r\|_{\bL^2(\dist{z}^{-\alpha},\Omega)}}
 \lesssim \mathfrak{E}_h.
 \label{eq:est_varphi}
\end{equation}

\emph{Step 6.} Collecting the estimates \eqref{eq:est_Phi_and_psi} and  \eqref{eq:est_varphi} we obtain the reliability bound \eqref{eq:reliability}. This concludes the proof.
\end{proof}

\subsection{Local efficiency estimates}

In this section, we analyze efficiency properties for the local error indicator $\mathcal{E}_K$ on the basis of standard bubble function arguments \cite{Verfurth}. Before proceeding with such analysis, we introduce the following notation: For an edge, triangle or tetrahedron $G$, let $\mathcal{V}(G)$ be the set of vertices of $G$. With this notation at hand, we introduce the following standard element and edge bubble functions. Let $K \in \T_h$ and $S \in \Sides$. We define
\begin{equation}
\label{eq:bubble_standard}
\Upsilon_{K}:=(d+1)^{d+1}\prod_{\mathrm{v}\in\mathcal{V}(K)}\lambda^{}_{\mathrm{v}|K},
\qquad
\Upsilon_{S}:=d^{d}_{}\prod_{\mathrm{v}\in\mathcal{V}(S)}\lambda_{\mathrm{v}|K} \textrm{ with } K\in\mathcal{N}_{S}.
\end{equation}
In these formulas, by $\lambda_{\mathrm{v}|K}$, we denote the barycentric coordinate function associated to $\mathrm{v} \in \mathcal{V}(K)$.

We will also make use of the following bubble functions, whose construction we owe to \cite[Section 5.2]{AGM}. Given $K \in \T_h$, we introduce $\Psi_K$, which satisfies  
$0 \leq \Psi_K \leq 1$,
\begin{equation}
\label{eq:bubble_T}
\Psi_K(z) = 0, \qquad |K| \lesssim \int_K \Psi_K, \qquad \| \GRAD \Psi_K \|_{\bL^{\infty}(R_K)} \lesssim h_K^{-1},
\end{equation}
and there exists a simplex $K^{*} \subset K$ such that $R_{K}:= \supp(\Psi_K) \subset K^{*}$. Notice that, since $\Psi_K$ satisfies \eqref{eq:bubble_T}, we have that, for every $m \in \mathbb{N}$
\begin{equation}
\label{eq:aux_bubble_T}
 \| \theta \|_{L^2(R_K)} \lesssim \| \Psi_K^{\frac{1}{2}} \theta \|_{L^2(R_K)} \quad \forall \theta \in \mathbb{P}_{m}(R_K),
\end{equation}
where the hidden constant depends on $m$, but does not depend on $\theta$ or $K$. Given $S \in \Sides$, we also introduce an edge bubble function $\Psi_S$, which satisfies $0 \leq \Psi_S \leq 1$,
\begin{equation}
\label{eq:bubble_S}
\Psi_S(z) = 0, \qquad |S| \lesssim \int_S \Psi_S, \qquad \| \GRAD \Psi_S \|_{ \bL^{\infty}(R_{S} ) } \lesssim h_S^{-1},
\end{equation}
and $R_S:= \supp(\Psi_S)$ is such that, if $\mathcal{N}_{S} = \{ K, K' \}$, there are simplices $K_{*} \subset K$ and $K_{*}' \subset K'$ such that $R_S \subset K_{*} \cup K_{*}' \subset K \cup K'$.

The following identities are essential to perform the upcoming local efficiency analysis. Invoke \eqref{eq:term_I}, \eqref{eq:RP_1}, and \eqref{eq:Xi} to arrive at
\begin{multline}
\label{eq:velocity_error_bubble}
\sum_{K \in \T_h} \int_{K} \Res_K \cdot \bv \diff x + \sum_{S \in \Sides} \int_S \Jump_S \cdot \bv \diff s 
=
\int_{\Omega} \nu \nabla \boldsymbol{\Phi} : \nabla \bv \diff x
=
\int_{\Omega} \left( \nu \nabla \ee_\ue: \nabla \bv   \right.
\\
\left.
-  \ee_{\pe} \DIV \bv
+(\ue\cdot\GRAD)\ee_\ue\cdot \bv + (\ee_\ue \cdot \GRAD)\ue_h \cdot \bv 
+ \frac12 \DIV \ee_\ue \ue_h \cdot \bv- \ee_{\Te}  \mathbf{g} \cdot \bv \right) \diff x
\end{multline}
for all $\bv\in \bH^1_0(\Omega)$. We recall the reader that $\Res_K$ and $\Jump_S$ are defined in \eqref{eq:residuals_navier}. Similarly, for every $r \in H_0^1(\dist{z}^{-\alpha},\Omega)$, we have
\begin{multline}
\label{eq:temperature_error_buble}
\displaystyle
\sum_{K \in \T_h} \int_{K} \mathfrak{R}_{K} r \diff x + \sum_{S \in \Sides} \int_S \mathfrak{J}_S r \diff s
\diff x + \langle \mathfrak{h} \delta_z , r \rangle = \int_{\Omega} \nabla \varphi \cdot \nabla r \diff x
\\
=\int_{\Omega} \left( \kappa \nabla \ee_{\Te} \cdot \nabla r -  \Te \ee_{\ue} \cdot \nabla r -  \ee_{\Te} \ue_h \cdot \nabla r \right)\diff x.
\end{multline}

We are now in position to provide the following local efficiency result.

\begin{theorem}[local efficiency]
Let $(\ue,\pe, \Te) \in \bH_0^1(\Omega) \times L_0^2(\Omega) \times H_0^1(\dist{z}^{\alpha},\Omega)$ be a solution to \eqref{eq:BusiWeak} and $(\ue_{h},\pe_{h}, \Te_h) \in \bX_h \times M_h \times W_h$ be its finite element approximation obtained as solution to \eqref{eq:Busih}. Let $\alpha \in (d-2,d)$. In the framework of Theorem \ref{thm:upper_bound_for_error}, we have the following local efficiency estimate:
\begin{equation}
\mathcal{E}_K \lesssim \interleave (\ee_\ue,\ee_\pe,\ee_{\Te}) \interleave_{\mathcal{S}_K},
\label{eq:local_efficiency}
\end{equation}
where the hidden constant is independent of the continuous and discrete solutions, the size of the elements in the mesh $\T_h$, and $\#\T_h$.
\end{theorem}
\begin{proof}
We examine each of the contributions of $\mathcal{E}_K$ separately so the proof involves several steps.

\emph{Step 1.} Let $K \in \T_h$. We bound the term $h_K^2 \| \Res_K \|^2_{\bL^2(K)}$ in \eqref{eq:estimator_Navier}. Invoke standard properties that the bubble function $\Upsilon_K$ satisfies to obtain the basic estimate
\begin{equation}
\label{eq:ResK_Theta}
\| \Res_K \|^2_{\bL^2(K)} 
= \int_{K} |\Res_K|^2 \diff x 
\lesssim  \int_{K} |\Res_K|^2 \Upsilon_K \diff x
= \int_{K} \Res_K \cdot \boldsymbol{\Theta}_K  \diff x,
\end{equation}
where $\boldsymbol{\Theta}_K := \Upsilon_K \Res_K $. Set $\bv = \boldsymbol{\Theta}_K$ as a test function on identity \eqref{eq:velocity_error_bubble}, utilize that $\supp \boldsymbol{\Theta}_K \subset K$, and standard inequalities to arrive at
\begin{multline}
\int_{K} \Res_K \cdot \boldsymbol{\Theta}_K  \diff x 
\leq 
\bigg(
\nu \| \nabla \ee_{\ue} \|_{\bL^2(K)} 
+ \| \ee_{\pe} \|_{L^2(K)}  + C_{e,1}g \|\GRAD \ee_{\Te} \|_{\bL^2(\dist{z}^\alpha,K)} 
\\
+ \frac{C_{4\to2}^2}{\nu} \| \GRAD \ue \|_{\bL^2(K)}\nu \|\GRAD \ee_\ue \|_{\bL^2(K)} 
+ \frac{3 C_{4\to2}^2}{2\nu} \| \GRAD \ue_h \|_{\bL^2(K)}\nu \|\GRAD \ee_\ue \|_{\bL^2(K)} 
\bigg)
\|\GRAD \boldsymbol{\Theta}_K   \|_{\bL^2(K)}.
\end{multline}
Next, we utilize the smallness assumption \eqref{eq:assumptions_for_Ritz} to obtain that
\begin{equation}
\label{eq:estimate_ResK_aux}
\int_{K} \Res_K \cdot \boldsymbol{\Theta}_K  \diff x  \lesssim \left( \nu \| \nabla \ee_{\ue} \|_{\bL^2(K)} + \| \ee_{\pe} \|_{L^2(K)} + \|\GRAD \ee_{\Te} \|_{\bL^2(\dist{z}^\alpha,K)} \right) \| \nabla \boldsymbol{\Theta}_K\|_{\bL^2(K)}.
\end{equation}
This, on the basis of  \eqref{eq:ResK_Theta} and the estimate $ \| \nabla \boldsymbol{\Theta}_K\|_{\bL^2(K)} \lesssim h_K^{-1} \| \Res_K\|_{\bL^2(K)}$, yields
\begin{equation}
\label{eq:estimate_ResK}
h_K^2 \| \Res_K \|^2_{\bL^2(K)} \lesssim  \nu^2 \| \nabla \ee_{\ue} \|^2_{\bL^2(K)} + \| \ee_{\pe} \|^2_{L^2(K)} + \|\GRAD \ee_{\Te} \|^2_{\bL^2(\dist{z}^\alpha,K)} .
\end{equation}

\emph{Step 2}. Let $K \in \T$ and $S \in \Sides_K$. We bound $h_K \|  \Jump_S  \|^2_{\bL^2(S)}$. Define $\boldsymbol{\Lambda}_S := \Upsilon_S \Jump_S$, where $\!\Jump_S$ and $\Upsilon_S$ are as in \eqref{eq:residuals_navier} and \eqref{eq:bubble_standard}, respectively. Basic properties of $\Upsilon_K$ yield
\begin{equation}
\label{eq:JumpS_Lambda}
\| \Jump_S \|^2_{\bL^2(S)} = \int_{S} |\Jump_S|^2 \diff s \lesssim  \int_{S} |\Jump_S|^2 \Upsilon_S \diff s
= \int_{S} \Jump_S \cdot \boldsymbol{\Lambda}_S  \diff s.
\end{equation}
Notice that $\supp \boldsymbol{\Lambda}_S \subset \mathcal{N}_S$. Setting $\bv = \boldsymbol{\Lambda}_S$ in \eqref{eq:velocity_error_bubble} yields
\begin{multline*}
\int_S \Jump_S \cdot \boldsymbol{\Lambda}_S \diff s 
= \sum_{K' \in \Ne_S} \int_{K'} 
\left( \nu \nabla \ee_\ue: \nabla \boldsymbol{\Lambda}_S 
-  \ee_{\pe} \DIV \boldsymbol{\Lambda}_S
+(\ue\cdot\GRAD)\ee_\ue\cdot \boldsymbol{\Lambda}_S 
\right.
\\
\left.
+ (\ee_\ue \cdot \GRAD)\ue_h \cdot \boldsymbol{\Lambda}_S + \frac12 \DIV \ee_\ue \ue_h \cdot \boldsymbol{\Lambda}_S - \ee_{\Te}  \mathbf{g} \cdot \boldsymbol{\Lambda}_S - \Res_{K'} \cdot \boldsymbol{\Lambda}_S 
\right) \diff x.
\end{multline*}
Similar arguments to the ones used to obtain \eqref{eq:estimate_ResK_aux} allow us to obtain
\begin{multline}
\label{eq:Jump_aux}
\int_S \Jump_S \cdot \boldsymbol{\Lambda}_S \diff s 
\lesssim 
\sum_{K' \in \Ne_S}  \| \Res_K \|_{\bL^2(K')}  \| \boldsymbol{\Lambda}_S \|_{\bL^2(K')}
\\
+ \sum_{K' \in \Ne_S} \left( \nu \| \nabla \ee_{\ue} \|_{\bL^2(K')} + \| \ee_{\pe} \|_{L^2(K')} + \|\GRAD \ee_{\Te} \|_{\bL^2(\dist{z}^\alpha,K')} \right)  \| \nabla \boldsymbol{\Lambda}_S \|_{\bL^2(K')}.
\end{multline}

On the other hand, by shape regularity, we have that
\begin{align*}
  \| \boldsymbol{\Lambda}_S \|_{\bL^2(K')} & \approx |K'|^{\frac{1}{2}} |S|^{-\frac{1}{2}} \| \boldsymbol{\Lambda}_S \|_{\bL^2(S)} \approx h_{K'}^{\frac{1}{2}} \| \boldsymbol{\Lambda}_S \|_{\bL^2(S)} \approx h_{K'}^{\frac{1}{2}} \| \Jump_S \|_{\bL^2(S)},
  \\
 \| \nabla \boldsymbol{\Lambda}_S \|_{\bL^2(K')} & \lesssim h_{K'}^{-1}  \| \boldsymbol{\Lambda}_S \|_{\bL^2(K')}\approx h_{K'}^{-\frac{1}{2}} \| \Jump_S \|_{\bL^2(S)}.
\end{align*}

In view of \eqref{eq:JumpS_Lambda}, the bound \eqref{eq:Jump_aux}, the estimates for $\boldsymbol{\Lambda}_S$ previously stated, and the estimate for $\| \Res_K \|_{\bL^2(K')}$ obtained in  \eqref{eq:estimate_ResK} we are capable of obtaining that
\begin{equation*}
\label{eq:estimate_JumpS}
h_K \| \Jump_S \|^2_{\bL^2(S)} \lesssim \sum_{K' \in \Ne_S} \left( \nu^2 \| \nabla \ee_{\ue} \|^2_{\bL^2(K')} + \| \ee_{\pe} \|^2_{L^2(K')} + \|\GRAD \ee_{\Te} \|^2_{\bL^2(\dist{z}^\alpha,K')} \right).
\end{equation*}

\emph{Step 3.} We now bound the residual term associated with the incompressibility constraint. Since $\DIV \ue = 0$, for any $K \in \T$, we have
\begin{equation*}
\label{eq:estimate_DIV}
\| \DIV  \ue_h  \|^2_{L^2(K)} = \| \DIV  \ee_{\ue}  \|^2_{L^2(K)}.
\end{equation*}

\emph{Step 4.} Let $K \in \T_h$. We bound $h_K^2 D_K^{\alpha} \| \mathfrak{R}_K  \|^2_{L^2(K)}$ in \eqref{eq:estimator_heat}. Define $\phi_K:=  \Psi_K \mathfrak{R}_K$, where $\mathfrak{R}_K$ and $\Psi_K$ are as in \eqref{eq:residuals_heat} and \eqref{eq:bubble_T}, respectively. Invoke \eqref{eq:aux_bubble_T} to arrive at
\begin{equation}
\label{eq:aux_mathfrakR}
 \| \mathfrak{R}_K  \|^2_{L^2(K)} \lesssim \int_{K} \mathfrak{R}_K \phi_K  \diff x.
\end{equation}
Set $r =\phi_K \in H_0^1(\dist{z}^{-\alpha},\Omega)$ as a test function in \eqref{eq:temperature_error_buble}. Utilize that $R_K = \supp \phi_K \subset K^{*} \subset K$, the property $\phi_K(z) = 0$, and the estimate \eqref{eq:boundedness_2} to obtain
\begin{multline}
\int_{K} \mathfrak{R}_{K} \phi_K \diff x 
\leq \left( \kappa \| \nabla \ee_{\Te}\|_{L^2(\dist{z}^{\alpha},K)}
+ C_{e,2} \| \nabla \ee_{\ue} \|_{\bL^2(K)}\| \nabla \Te \|_{\bL^2(\dist{z}^{\alpha},K)}
\right.
\\
\left.
+
C_{e,2} \| \nabla \ue_h \|_{\bL^2(K)}\| \nabla \ee_{\Te} \|_{\bL^2(\dist{z}^{\alpha},K)} \right) \| \nabla \phi_K \|_{\bL^2(\dist{z}^{-\alpha},K)}.
\label{eq:estimate_frakResK_aux}
\end{multline}
Invoke \eqref{eq:aux_mathfrakR}, the estimate $\| \nabla \phi_K \|_{\bL^2(\dist{z}^{-\alpha},K)} \lesssim h_K^{-1}D_K^{-\alpha/2} \| \mathfrak{R}_{K} \|_{\bL^2(K)}$ \cite[Proposition 8]{MR3892359}, and the smallness assumption \eqref{eq:assumptions_for_Ritz} to conclude that
\begin{equation}
\label{eq:estimate_frakResK}
h_K^2 D_K^{\alpha}\| \mathfrak{R}_K \|^2_{\bL^2(K)} \lesssim 
 (\kappa^2+\nu^2) \| \nabla \ee_{\Te}\|^2_{L^2(\dist{z}^{\alpha},K)} + \nu^2 \| \nabla \ee_{\ue} \|^2_{\bL^2(K)} .
\end{equation}

\emph{Step 5.} Let $K \in \T_h$ and $S \in \Sides_K$. The bound of  $h_K D_K^{\alpha} \| \mathfrak{J}_S \|^2_{L^2(S)}$ follows similar arguments as the ones developed in Step 2 upon utilizing \eqref{eq:estimate_frakResK}. In fact, we have
\begin{equation*}
\label{eq:estimate_JumpfrakS}
h_K D_K^{\alpha} \| \mathfrak{J}_S \|^2_{L^2(S)} \lesssim \sum_{K' \in \Ne_S} \left( 
(\nu^{2}+\kappa^{2})\| \nabla \ee_{\Te}\|^2_{L^2(\dist{z}^{\alpha},K')} + \nu^{2} \| \nabla \ee_{\ue} \|^2_{\bL^2(K')}
\right).
\end{equation*}

\emph{Step 6.} We now control the term $|\mathfrak{h}|h_K^{\alpha+2-d} \# (\{z\} \cap K)$ in \eqref{eq:estimator_heat}. Let $K\in\T_h$, and notice first that, if $T\cap\{z\}=\emptyset$, then estimate \eqref{eq:local_efficiency}, follows from the estimates derived in the previous steps. If, on the other      hand, $K\cap\{z\}=\{z\}$, then we must obtain a bound for the term $|\mathfrak{h}|h_K^{\alpha+2-d}$. To do so we follow the arguments developed in the proof of \cite[Theorem 5.3]{AGM} that yield the existence of a smooth $\eta$ such that
\begin{equation*}\label{eq:morin}
\eta(z)=1,\quad
\|\eta\|_{L^{\infty}(\Omega)}=1,\quad
\|\nabla\eta\|_{L^{\infty}(\Omega)}=h^{-1}_{K},\quad
\Omega_{\eta}:=\textrm{supp}(\eta)\subset\mathcal{S}_{K},
\end{equation*}
where $\mathcal{S}_{K}$ is defined in \eqref{eq:NeTstar}. With this function at hand, define $r_{\eta}:=\mathfrak{h}\eta\in W^{1,\infty}_0(\Omega) \subset H_0^1(\dist{z}^{-\alpha},\Omega)$ and notice that
\begin{equation*}
\begin{aligned}
 |\mathfrak{h}|^{2} 
 =  
\langle \mathfrak{h}\delta_{z},r_{\eta} \rangle
& = 
\int_{\Omega}(\kappa\nabla \Te \cdot\nabla r_{\eta}-\Te \ue\cdot\nabla r_{\eta})\diff x 
\\
& = \int_{\Omega} \nabla \varphi \cdot \nabla r_{\eta} \diff x
+ \int_{\Omega}(\kappa\nabla \Te_h \cdot\nabla r_{\eta}- \Te_h \ue_{h}\cdot\nabla r_{\eta})\diff x,
\end{aligned}
\end{equation*}
where we have used equation \eqref{eq:RP_3}. We thus apply similar arguments to the ones that led to \eqref{eq:estimate_frakResK_aux}, integration by parts, and basic estimates to arrive at
\begin{multline*}
 |\mathfrak{h}|^{2}  \lesssim \left(
(\kappa+\nu)\|\nabla \ee_{\Te}\|_{\bL^2(\dist{z}^{\alpha},\mathcal{S}_{K})} + \nu\|\nabla \ee_{\ue}\|_{\bL^2(\mathcal{S}_K)}
\right)
\|\nabla r_{\eta}\|_{\bL^2(\dist{z}^{-\alpha},\mathcal{S}_K)} 
\\
+
\sum_{K'\in \T_h : K' \subset \mathcal{S}_K} 
\left(
\| \mathfrak{R}_K \|^{}_{\bL^2(K')}
\|r_{\eta}\|_{L^2(K')} 
+
\sum_{S \in \Sides_{K'}: S \not\subset\partial \mathcal{S}_K}
\| \mathfrak{J}_S \|^{}_{L^2(S)}
\|r_{\eta}\|_{L^2(S)}
\right),
\end{multline*}
where we have also used the the smallness assumption \eqref{eq:assumptions_for_Ritz}. Using the shape regularity of the mesh, in conjunction with the fact that, since $z\in K$, $h_{K}\approx D_{K}$, the bounds
\[
\|\nabla\eta\|_{\bL^2(\dist{z}^{-\alpha},\mathcal{S}_K)}
\lesssim
h_{K}^{\frac{d-2}{2}-\frac{\alpha}{2}},
\quad
\|\eta\|_{L^2(\mathcal{S}_K)}
\lesssim
h_{K}^{\frac{d}{2}},
\quad
\|\eta\|_{L^2(S)}
\lesssim
h_{K}^{\frac{d-1}{2}},
\]
allow us to conclude that 
\begin{gather*}
|\mathfrak{h}|^{}
\lesssim 
h_{K}^{\frac{d-2}{2}-\frac{\alpha}{2}}
\left(
(\kappa+\nu)\|\nabla \ee_{\Te}\|^{}_{\bL^2(\dist{z}^{\alpha},{\mathcal{S}_{K}})} + 
\nu^{}\|\nabla \ee_{\ue}\|^{}_{\bL^2(\dist{z}^{\alpha},{\mathcal{S}_{K}})}
\right)
\\
+
\sum_{K'\in \T_h : K' \subset \mathcal{S}_K} 
h_{K'}^{\frac{d-2}{2}-\frac{\alpha}{2}}
\Big(
h_{K'}D^{\frac{\alpha}{2}}_{K'}
\| \mathfrak{R}_K \|^{}_{\bL^2(K')}
+
\sum_{S \in \Sides_{K'}: S \not\subset\partial \mathcal{S}_K}
h_{K'}^{\frac{1}{2}}D^{\frac{\alpha}{2}}_{K'}
\| \mathfrak{J}_S \|^{}_{L^2(S)}
\Big).
\end{gather*}

Finally, combining all the previous results \eqref{eq:local_efficiency} follows.
\end{proof}


\section{Numerical experiments}
\label{sec:numerics}

In this section we conduct a series of numerical examples that illustrate the performance of the a posteriori error estimator we have devised and analyzed in section~\ref{sec:aposteriori}. The examples have been carried out with the help of a code that we implemented using \texttt{C++}. All matrices have been assembled exactly and global linear systems were solved using the multifrontal massively parallel sparse direct solver (MUMPS) \cite{MR1856597,MR2202663}. The \emph{element} and \emph{interelement} residuals are computed with the help of quadrature formulas which are exact. To visualize finite element approximations we have used the open--source application ParaView \cite{Ahrens2005ParaViewAE,Ayachit2015ThePG}. 

For a given partition $\T_h$, we solve \eqref{eq:Busih} with the discrete spaces $\mathbf{X}_h$, $M_h$, and $W_h$ given by \eqref{eq:th_V}, \eqref{eq:th_P}, and the space of continuous piecewise polynomial functions of degree two, respectively. To be precise, to adaptively solve the nonlinear system  \eqref{eq:Busih} we proceed as in Algorithm \ref{Algorithm}. 
We comment that, in Algorithm \ref{Algorithm-Fixed-Point}, for an initial partition $\T_{0}$, the initial guesses $\Te_{h}^{0} \in W_{h}$ and $(\ue_h^0,\pe_h^0)\in  \bX_h\times W_h$ are obtained as the respective solutions to the following problems:
\begin{equation*}
\int_\Omega \kappa \GRAD \Te_h^{0} \cdot \GRAD r_h~ \diff x = \langle \calH, r_h \rangle \quad \forall r_h \in W_h
\end{equation*}
and 
\begin{equation*}
\int_\Omega  \nu \GRAD \ue_h^{0} : \GRAD \bv_h -\int_\Omega \pe_h^{0} \DIV \bv_h \diff x = \int_\Omega \Te_h^{0}\bg \cdot \bv_h \diff x,
\qquad 
    \int_\Omega q_h \DIV \ue_h^0 \diff x = 0,
\end{equation*}
for all $\bv_h \in \bX_h$ and  $q_h \in M_h$, respectively. Once the discrete solution is obtained, we compute, for all $K\in \T_{h}$, the local a posteriori error indicators $\mathcal{E}_{K}$, defined in Section \ref{sub_section_estimator}, to drive the adaptive mesh refinement
procedure described in Algorithm \ref{Algorithm}. 
A sequence of adaptively refined meshes is thus generated from an initial mesh.
\footnotesize{
\begin{algorithm}[h!]
\caption{\textbf{Adaptive Algorithm}}
\label{Algorithm}
Input: Initial mesh $\T_{0}$, interior point $z\in\Omega$, and parameter $\nu$, $\kappa$, $\mathfrak{h}$, and $\alpha\in(0,2)$;\\
\textbf{1:} Solve the discrete problem \eqref{eq:Busih} by using \textbf{Algorithm} \ref{Algorithm-Fixed-Point}; \\  
\textbf{2:} For each $K \in \T_h$ compute the local error indicators $\mathcal{E}_K$ defined in Section \ref{sub_section_estimator};
\\
\textbf{3:} Mark an element $K \in \T_h$ for refinement if 
\vspace*{-0.3cm}
\[
  \mathcal{E}_K > \frac12 \max_{K' \in \T_h} \mathcal{E}_{K'};
  \vspace*{-0.3cm}
\]
\textbf{4:} From step $\boldsymbol{3}$, construct a new mesh, using a longest edge bisection algorithm \cite{MR1329875}. Set $i \leftarrow i + 1$, and go to step $\boldsymbol{1}$.
\end{algorithm}}
\footnotesize{
\begin{algorithm}[h!]
\caption{\textbf{Fixed-Point Algorithm}}
\label{Algorithm-Fixed-Point}
Input: Initial guess $(\ue_{h}^{0},\pe_{h}^{0},\Te_{h}^{0})\in \bX_h\times M_h\times W_{h}$ and $\textrm{tol}=10^{-8}$.\\
\textbf{1:} For $i\geq 0$, find $(\bu_{h}^{i+1},p_{h}^{i+1})\in \bX_h\times M_h$ such that \vspace*{-0.3cm}
\begin{align*}
&\int_\Omega \left( \nu \GRAD \ue_h^{i+1} : \GRAD \bv_h + (\ue_h^{i}\cdot\GRAD) \ue_h^{i+1} \cdot \bv_h + \frac12 \DIV \ue_h^{i} \ue_h^{i+1} \cdot \bv_h \right) \diff x 
       -\int_\Omega \pe_h^{i+1} \DIV \bv_h \diff x \\
& =
 \int_\Omega \Te_h^{i}\bg \cdot \bv_h \diff x\quad \forall \bv_h \in \bX_h.
\end{align*}
Then, $\Te_h^{i+1} \in W_h$ is found as the solution of
\begin{equation*}
\int_\Omega \left( \kappa \GRAD \Te_h^{i+1} \cdot \GRAD r_h - \Te_h^{i+1} \ue_h^{i+1} \cdot \GRAD r_h \right) \diff x = \langle \mathfrak{h}\delta_{z}, r_h \rangle \quad \forall r_h \in W_h.
\end{equation*}
\textbf{2:} If $\|(\ue_{h}^{i+1},\pe_{h}^{i+1},\Te_{h}^{i+1})-(\ue_{h}^{i},\pe_{h}^{i},\Te_{h}^{i})\|^{}_{2}>\textrm
{tol}$, set $i \leftarrow i + 1$, and go to step $\boldsymbol{1}$. Otherwise, \textbf{return} $(\ue_{h},\pe_{h},\Te_{h})=(\ue_{h}^{i+1},\pe_{h}^{i+1},\Te_{h}^{i+1})$.
\end{algorithm}}
\normalsize

Finally, we denote the total number of degrees of freedom by $\textsf{Ndof} = \dim(\bX_h)+\dim(M_h)+\dim(W_{h})$.

We now explore the performance of the devised a posteriori error estimator in two problems with homogeneous Dirichlet boundary conditions on convex and non--convex domains. In all the numerical experiments we have considered $\nu=\kappa=1$, $\bg=[1,0]^{\intercal}$, $\mathfrak{h}=1$,  $z=(0.5,0.5)$, and different values for the exponent of the Muckenhoupt weight: $\alpha\in\{0.1,0.5,1.0,1.5,1.9\}$. We let
\begin{enumerate}[(i)]
\item $\Omega=(0,1)^2$, for example 1, and
\item $\Omega=(-1,1)^{2}\setminus[0,1)\times[-1,0)$, for example 2.
\end{enumerate}

In Figures \ref{fig_1} and \ref{fig_3} we present, within the setting of examples 1 and 2, respectively, experimental rates of convergence for the error estimator $\mathcal{E}_h$. We also present the initial meshes used in the adaptive algorithm. We observe that optimal experimental rates of convergence are attained for all the values of the parameter $\alpha$ that we have considered. We also observe that a better value of the estimator, at a fixed mesh, can be obtained for values of $\alpha$ closer to two. We notice that, when $\alpha$ is small, after a certain number of adaptive iterations, there are elements $K$ around $z$ such that $|K| \approx 10^{-16}$. This makes impossible more computations within the adaptive procedure.

In Figures \ref{fig_2} and \ref{fig_4} we present, for examples 1 and 2, respectively, a series of meshes obtained after 30 adaptive iterations.
We observe that most of the refinement is concentrated around the singular source point. For the case of example 2, and after 30 adaptive refinements, the adaptive loop also concentrates the refinement around the reentrant corner when $\alpha\in[1,2)$.

Finally, in Figure \ref{fig_5} we present, for the setting of Example 2 with $\alpha=1.5$, $|\ue_h|$, its associated streamlines, the pressure $\pe_h$, and the temperature $\Te_h$ over a mesh containing 16105 elements and 8178 vertices; the latter being obtained after 65 iterations of our adaptive loop.

\begin{figure}[ht]
\begin{minipage}[c]{0.5\linewidth}
\centering
$\mathcal{E}_h$\\
\psfrag{alpha=0.1}{$\alpha=0.1$}
\psfrag{alpha=0.5}{$\alpha=0.5$}
\psfrag{alpha=1.0}{$\alpha=1.0$}
\psfrag{alpha=1.5}{$\alpha=1.5$}
\psfrag{alpha=1.9}{$\alpha=1.9$}
\psfrag{rate(h2)}{$\mathsf{Ndof}^{-1}$}
\includegraphics[trim={0 0 0 0},clip,width=4cm,height=4cm,scale=0.6]{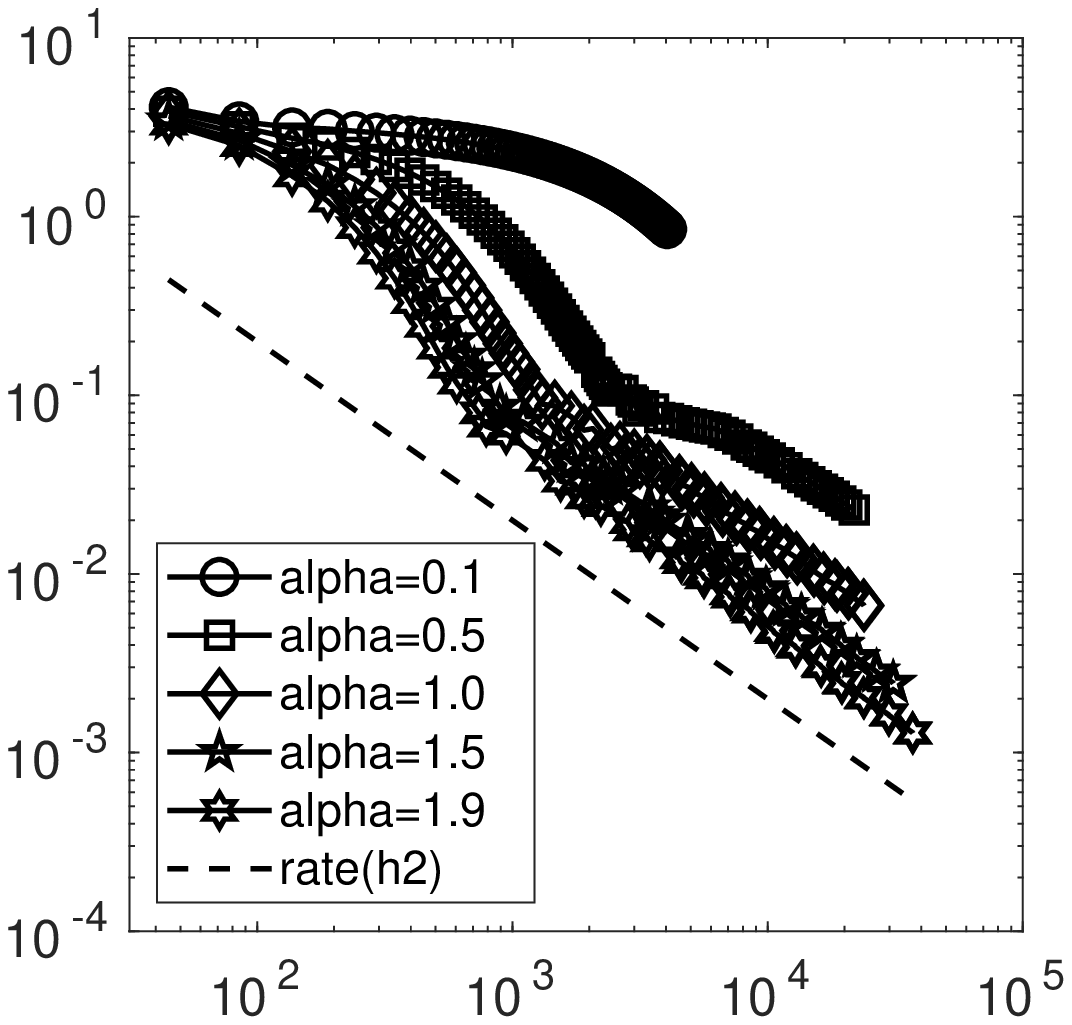}
\end{minipage}
\begin{minipage}[c]{0.5\linewidth}
\centering
\includegraphics[trim={0 0 0 0},clip,width=6cm,height=4cm,scale=1.0]{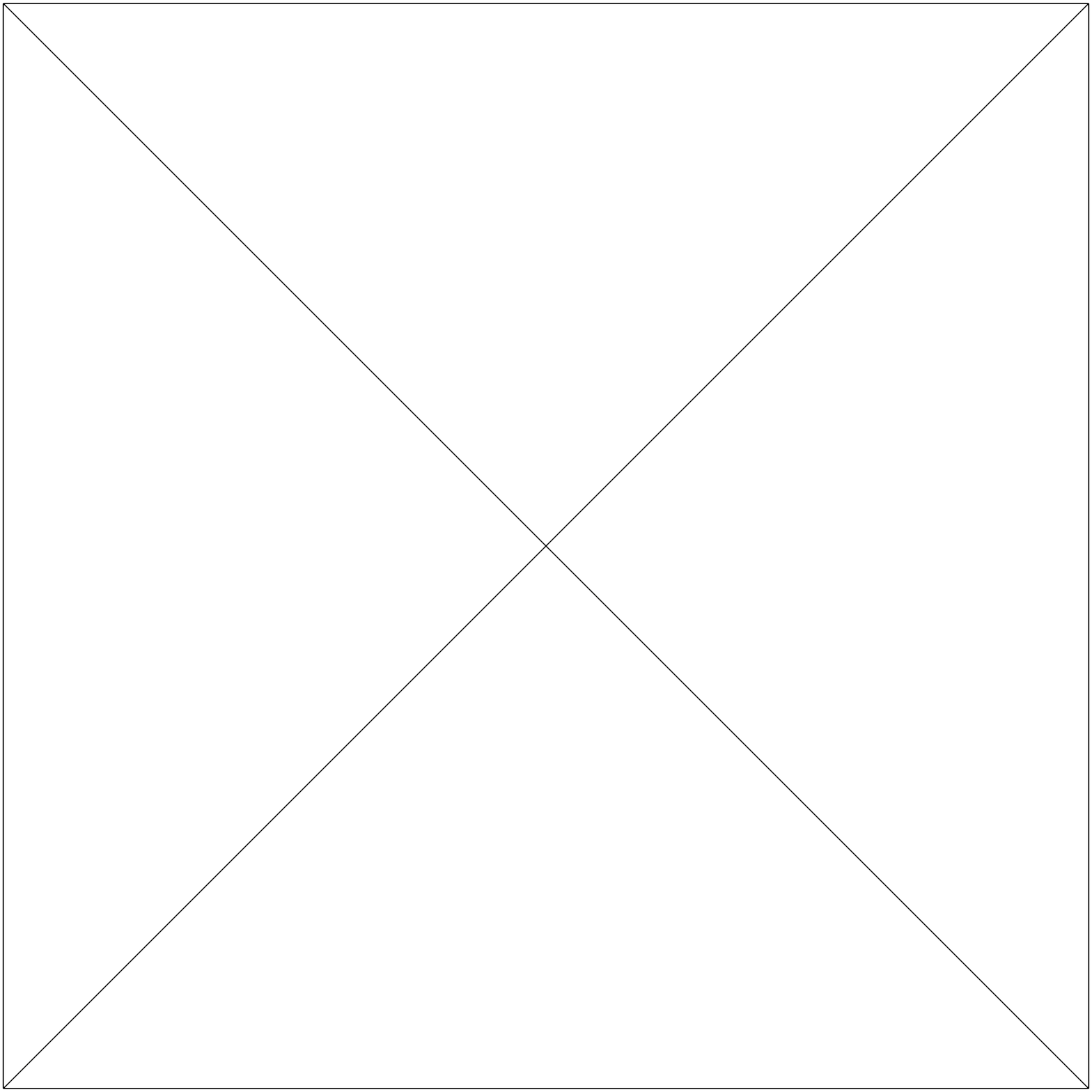}
\end{minipage}
\caption{Example 1: Experimental rates of convergence for the error estimator $\mathcal{E}_h$ considering $\alpha\in\{0.1,0.5,1.0,1.5,1.9\}$ (left) and the initial mesh used in the adaptive algorithm (right).}
\label{fig_1}
\end{figure}

\begin{figure}[ht]
\begin{minipage}[c]{0.32\linewidth}
\centering
$\alpha=0.1$
\includegraphics[trim={30cm 0 20cm 0},clip,width=4cm,height=3.7cm,scale=1.0]{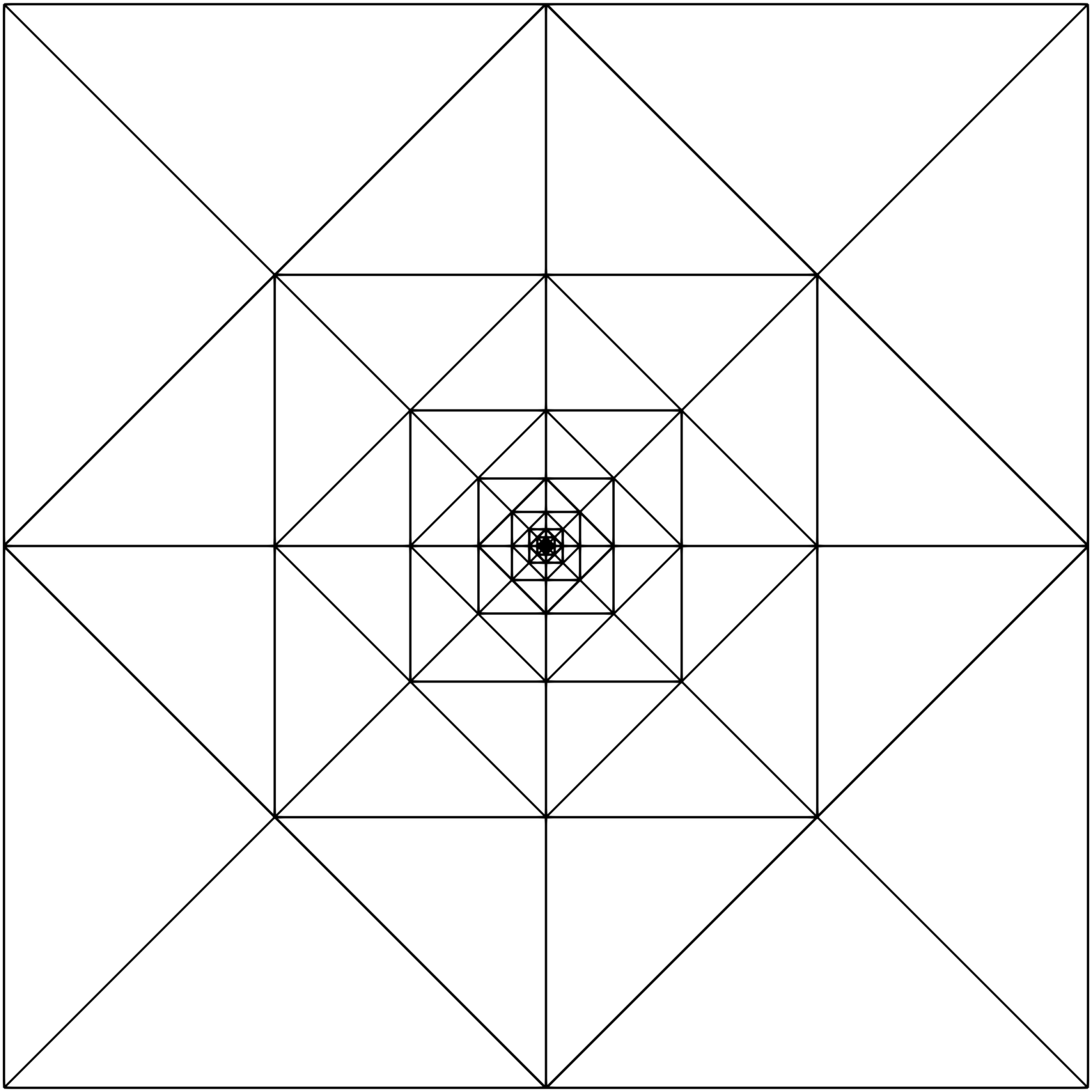}\\
(a)
\end{minipage}
\begin{minipage}[c]{0.32\linewidth}
\centering
$\alpha=0.5$
\includegraphics[trim={30cm 0 20cm 0},clip,width=4cm,height=3.7cm,scale=1.0]{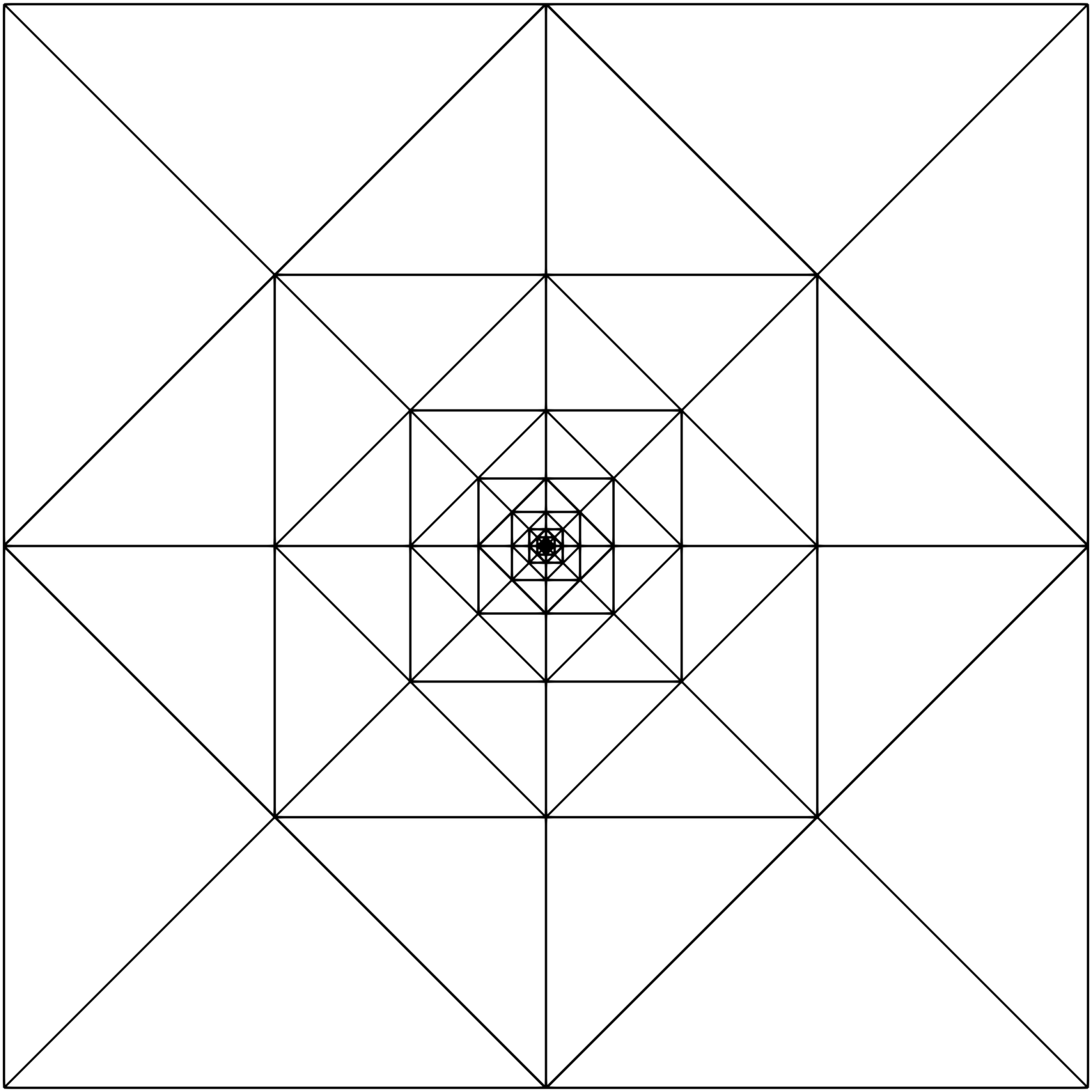}\\
(b)
\end{minipage}
\begin{minipage}[c]{0.35\linewidth}
\centering
$\alpha=1.0$
\includegraphics[trim={30cm 0 20cm 0},clip,width=4cm,height=3.7cm,scale=1.0]{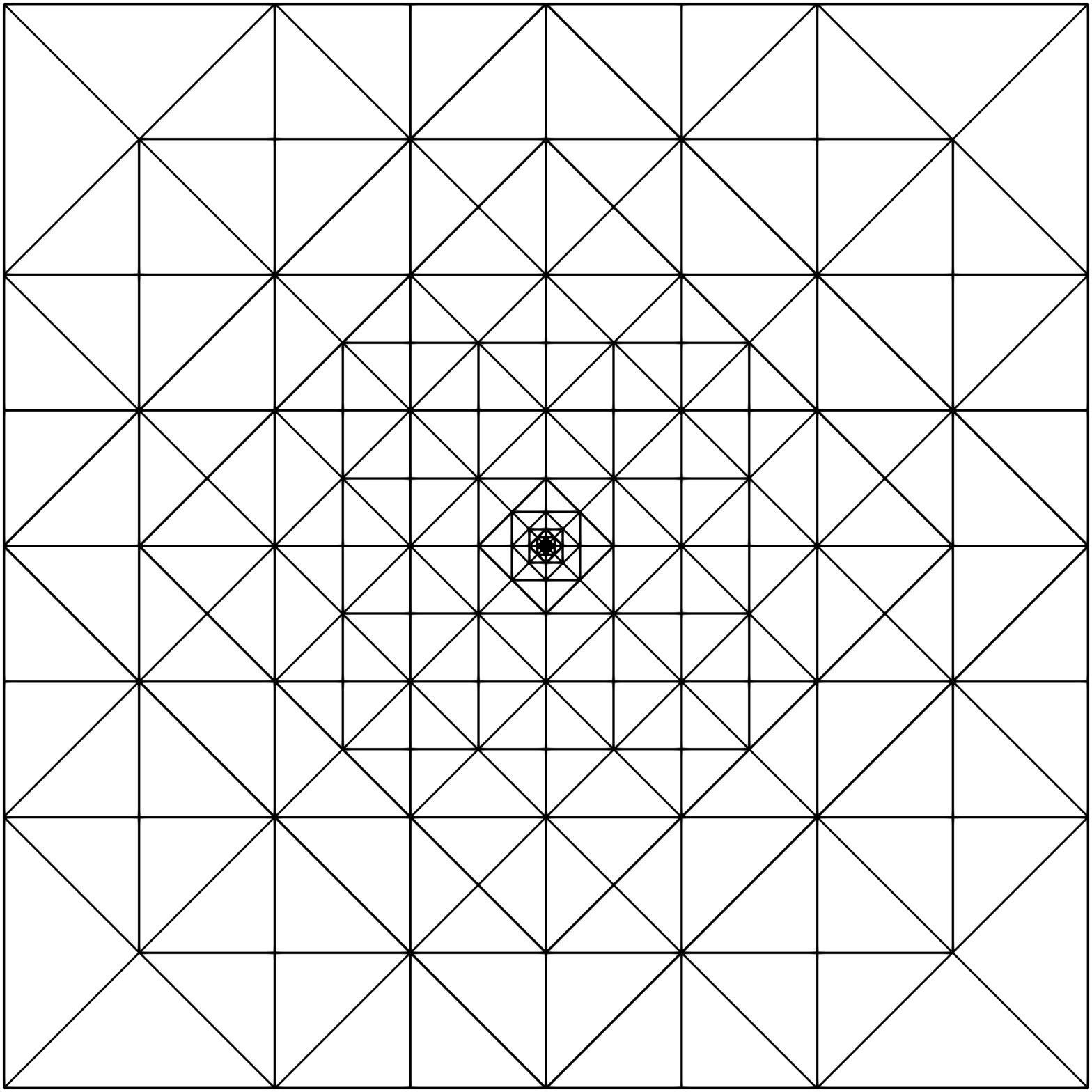}\\
(c)
\end{minipage}
\\
\centering
\begin{minipage}[c]{0.4\linewidth}
\centering
$\alpha=1.5~~~~~~$\\
\includegraphics[trim={30cm 0 20cm 0},clip,width=4cm,height=3.7cm,scale=1.0]{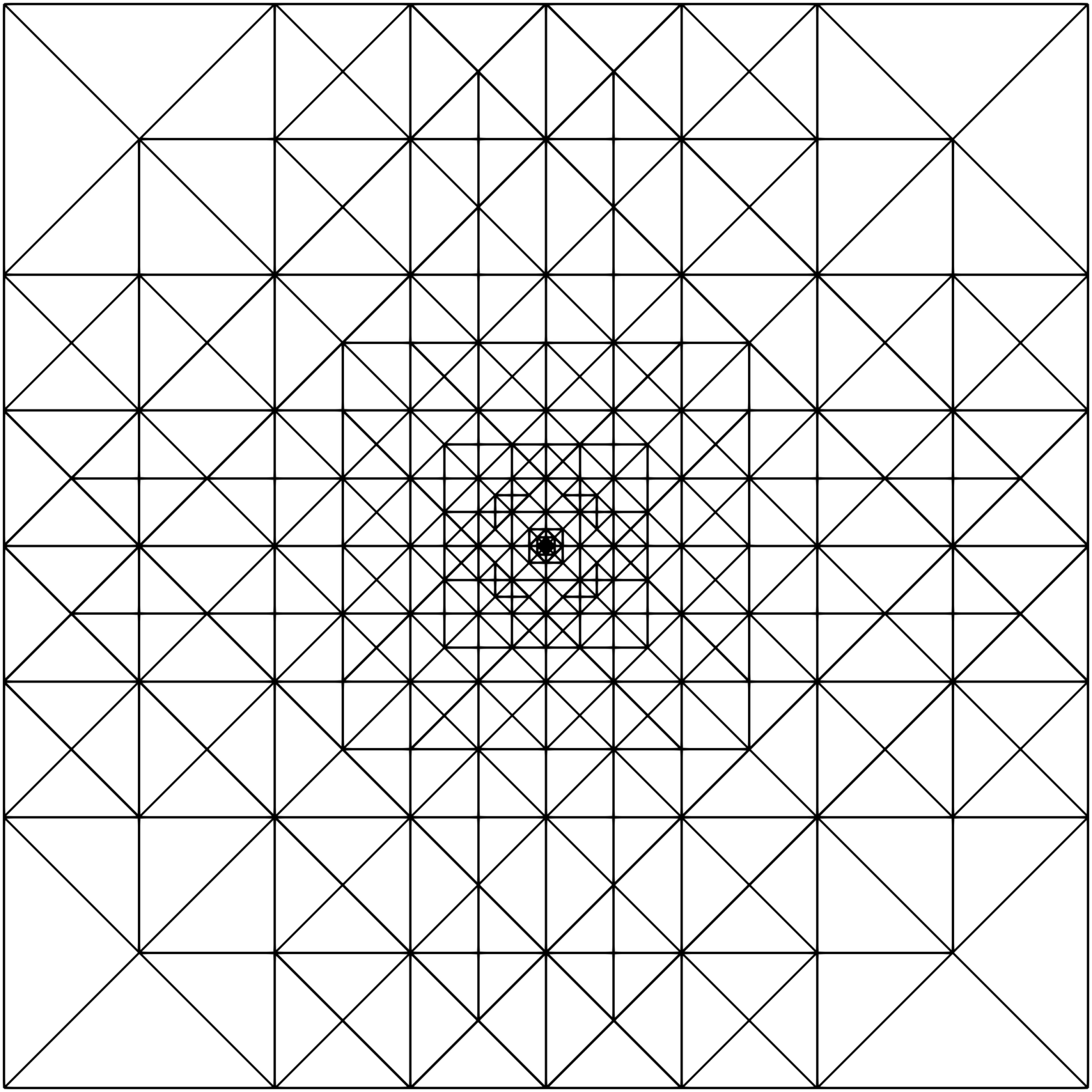}\\
$\textrm{(d)}~~~~~$
\end{minipage}
\begin{minipage}[c]{0.45\linewidth}
\centering
$\alpha=1.9~~~~~~$\\
\includegraphics[trim={30cm 0 20cm 0},clip,width=4cm,height=3.7cm,scale=1.0]{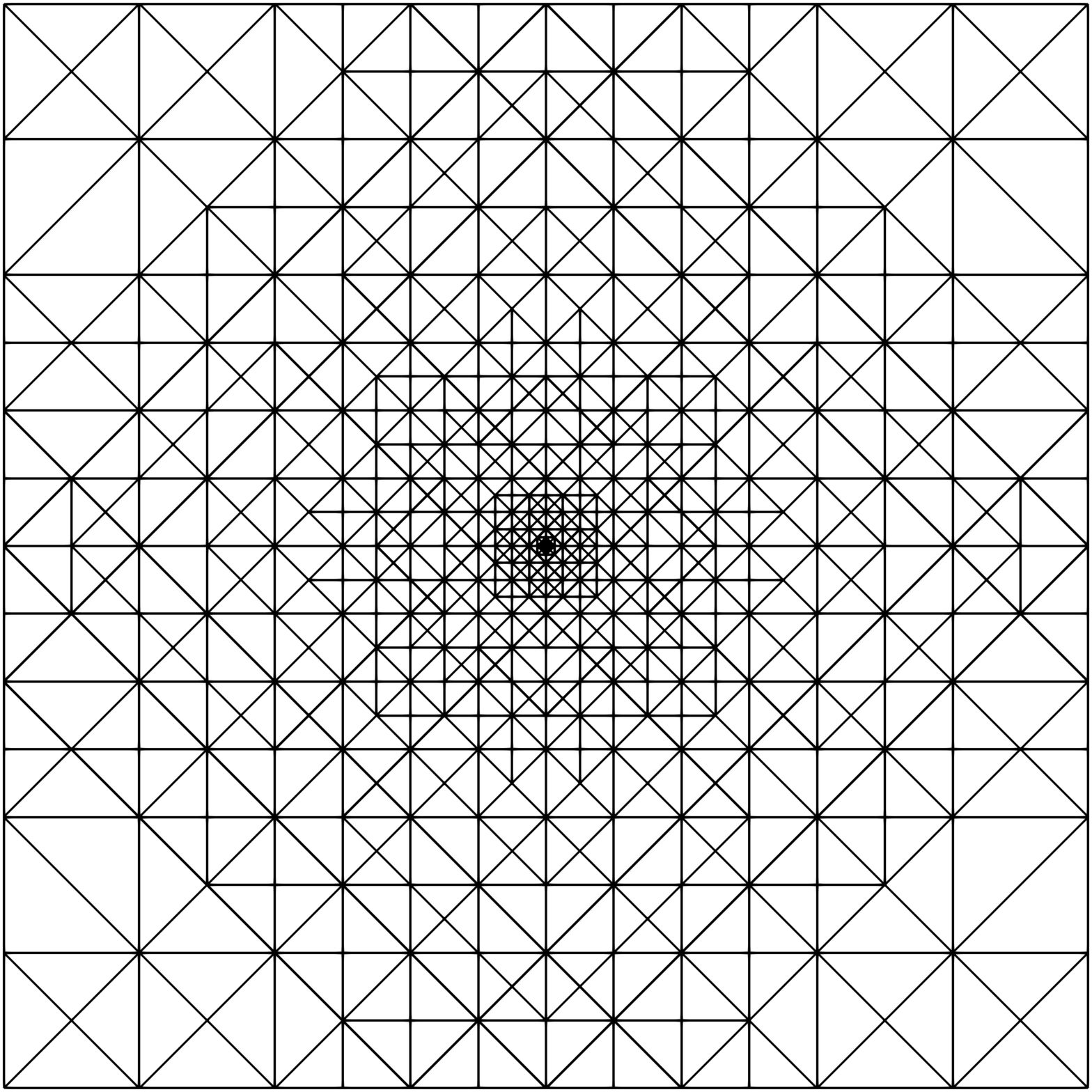}\\
(e)
\end{minipage}
\caption{Example 1: Meshes obtained after 30 iterations of our adaptive loop for 
(a) $\alpha = 0.1$ (232 elements and 121 vertices); (b) $\alpha = 0.5$ (232 elements and 121 vertices); (c) $\alpha = 1.0$ (392 elements and 209 vertices); (d) $\alpha = 1.5$ (592 elements and 309 vertices); and (e) $\alpha = 1.9$ (1056 elements and 553 vertices).}
\label{fig_2}
\end{figure}

\begin{figure}[ht]
\centering
\begin{minipage}[c]{0.45\linewidth}
\centering
$\mathcal{E}_h$\\
\psfrag{alpha=0.1}{$\alpha=0.1$}
\psfrag{alpha=0.5}{$\alpha=0.5$}
\psfrag{alpha=1.0}{$\alpha=1.0$}
\psfrag{alpha=1.5}{$\alpha=1.5$}
\psfrag{alpha=1.9}{$\alpha=1.9$}
\psfrag{rate(h2)}{$\mathsf{Ndof}^{-1}$}
\includegraphics[trim={0 0 0 0},clip,width=3.5cm,height=4cm,scale=0.6]{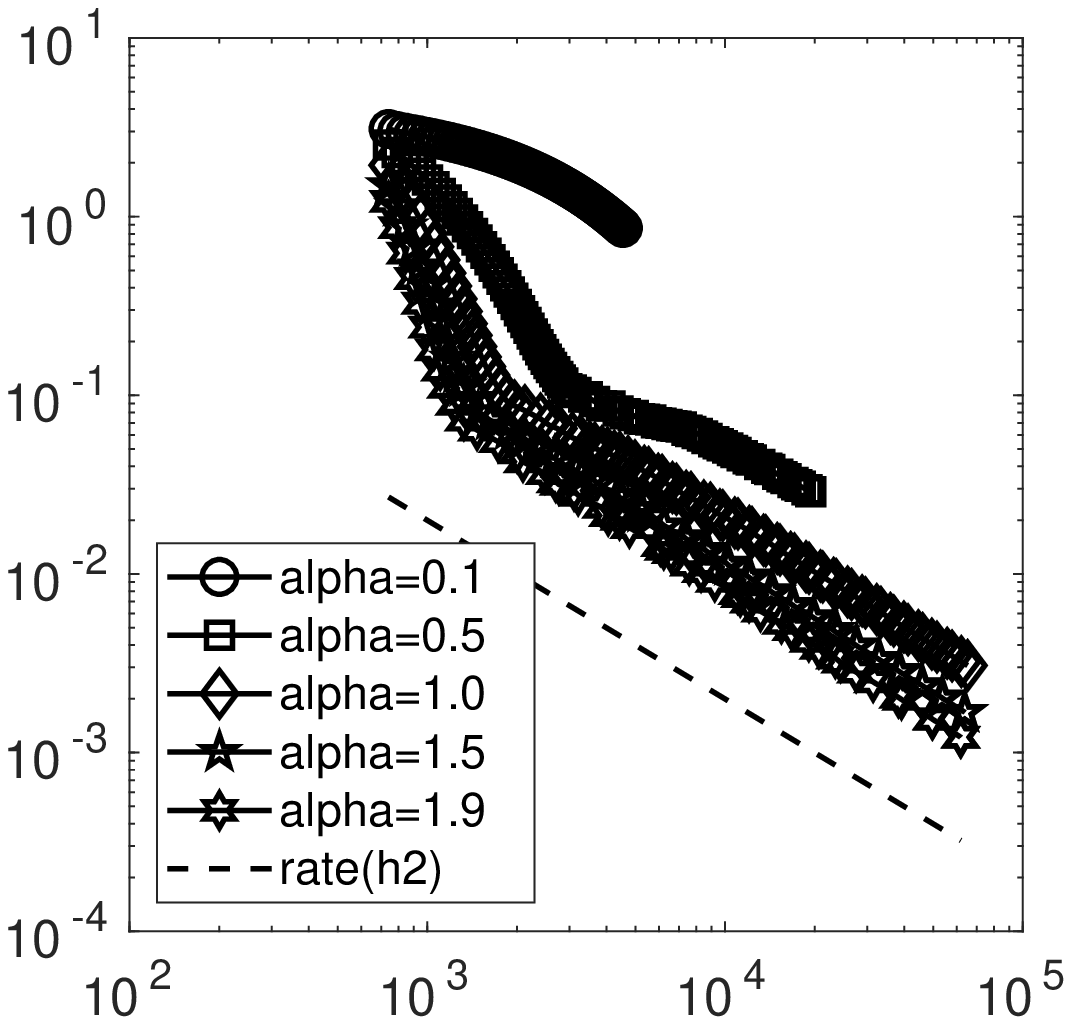}
\end{minipage}
\begin{minipage}[c]{0.45\linewidth}
\centering
\includegraphics[trim={30cm 0 10cm 0},clip,width=4.5cm,height=4cm,scale=1.0]{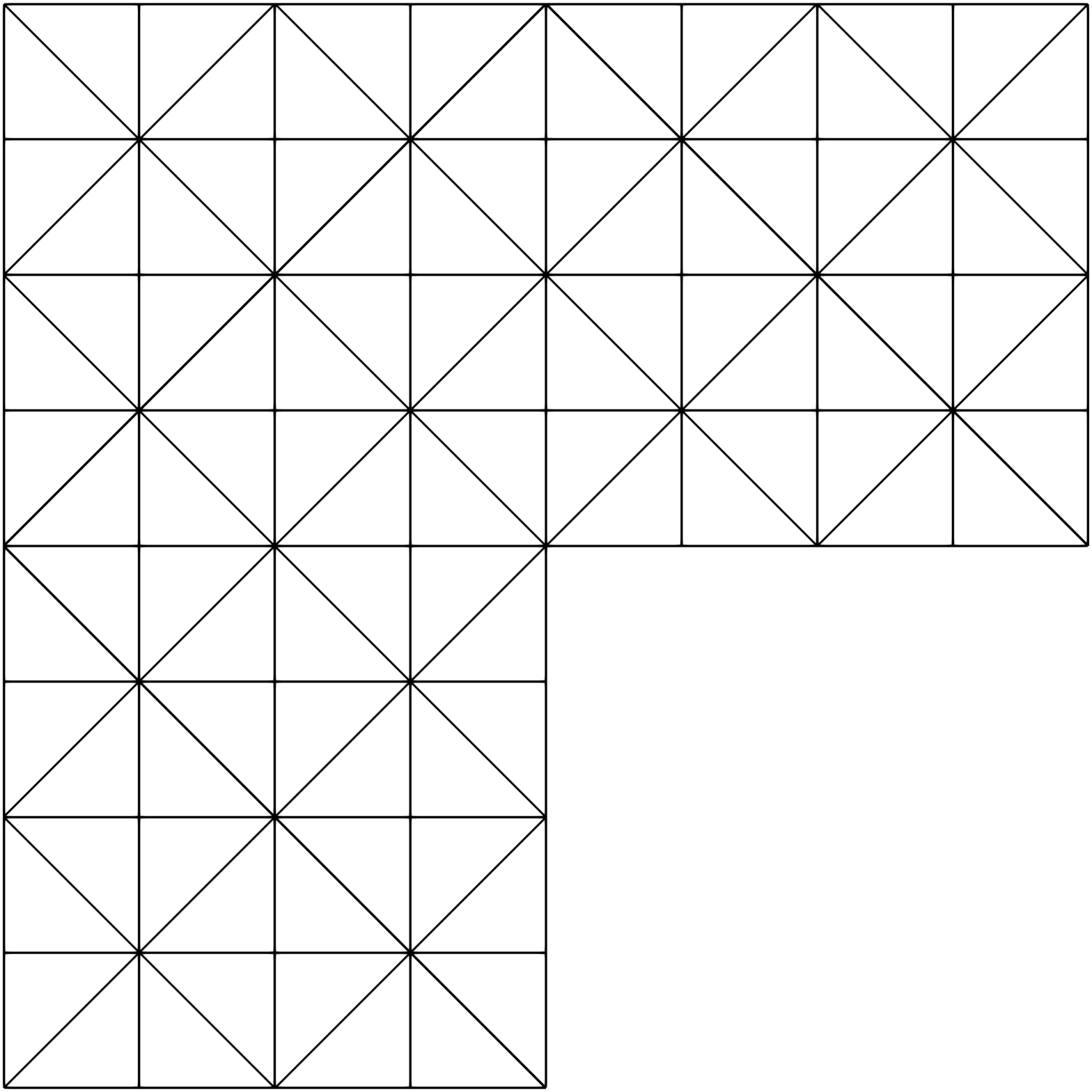}
\end{minipage}
\caption{Example 2: Experimental rates of convergence for the error estimator $\mathcal{E}_h$ considering $\alpha\in\{0.1,0.5,1.0,1.5,1.9\}$ (left) and the initial mesh used in the adaptive algorithm, which contains 96 elements and 65 vertices (right).}
\label{fig_3}
\end{figure}

\begin{figure}[ht]
\begin{minipage}[c]{0.32\linewidth}
\centering
$\alpha=0.1$
\includegraphics[trim={30cm 0 20cm 0},clip,width=4cm,height=3.7cm,scale=1.0]{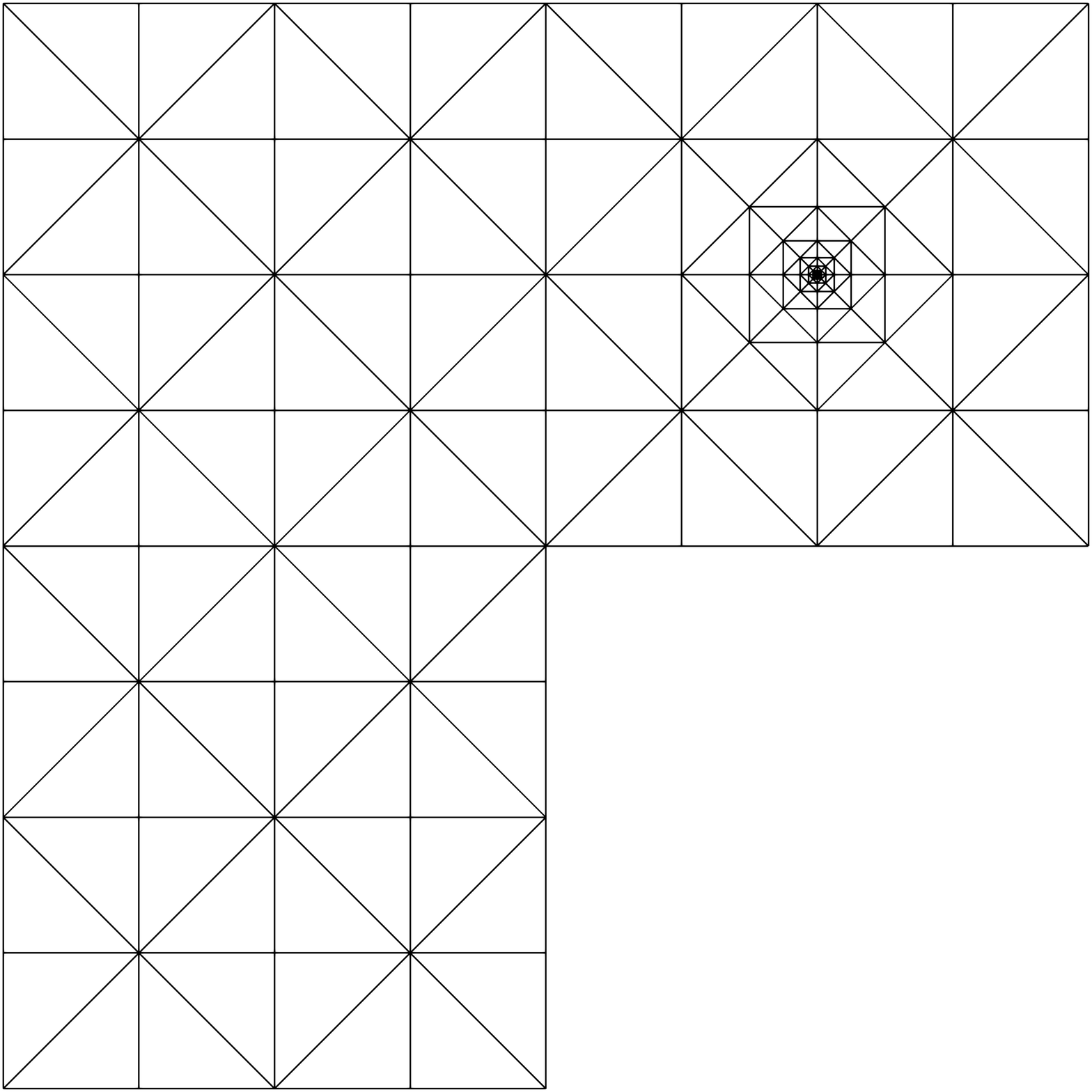}\\
(a)
\end{minipage}
\begin{minipage}[c]{0.32\linewidth}
\centering
$\alpha=0.5$
\includegraphics[trim={30cm 0 20cm 0},clip,width=4cm,height=3.7cm,scale=1.0]{lshape_1.eps}\\
(b)
\end{minipage}
\begin{minipage}[c]{0.35\linewidth}
\centering
$\alpha=1.0$
\includegraphics[trim={30cm 0 20cm 0},clip,width=4cm,height=3.7cm,scale=1.0]{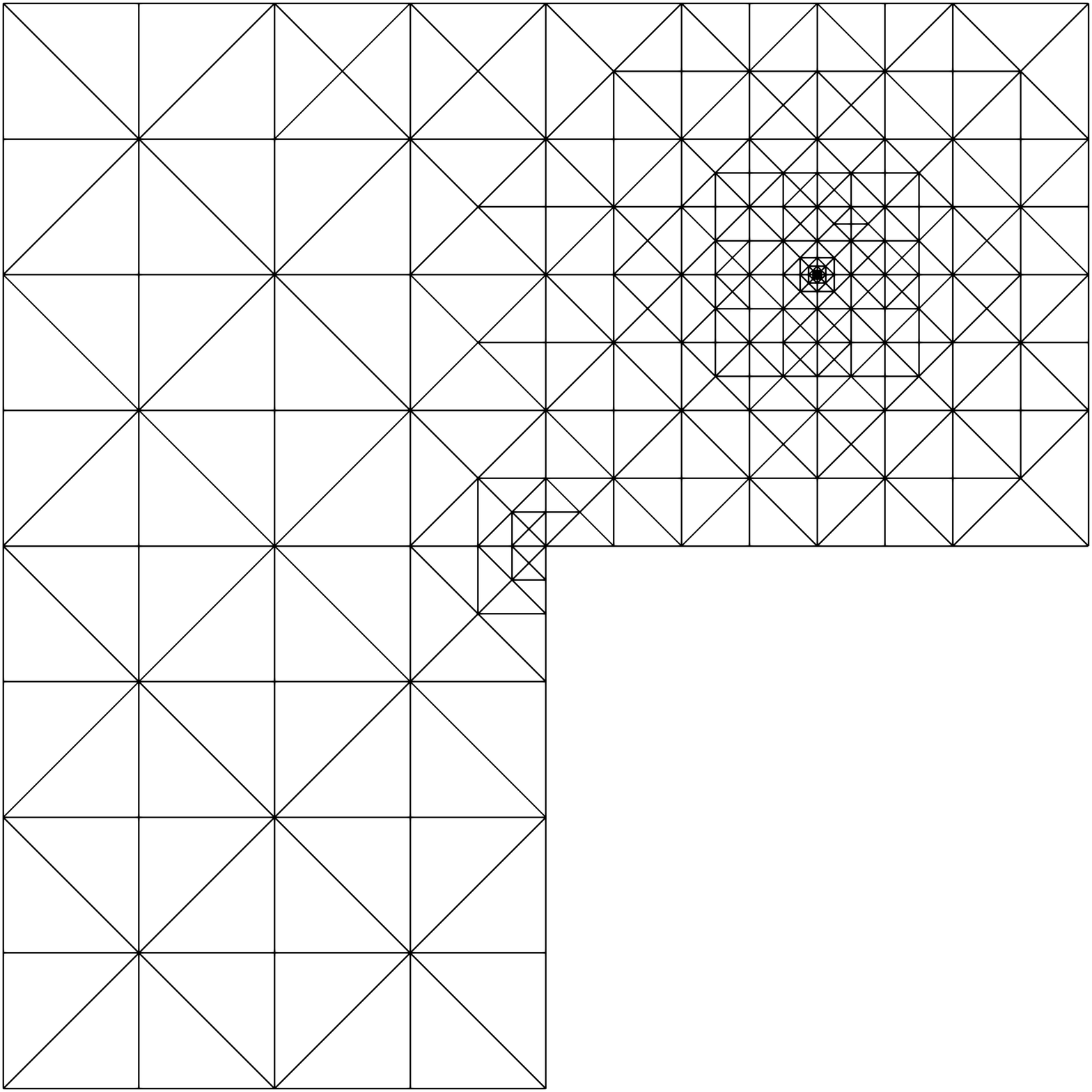}\\
(c)
\end{minipage}
\\
\centering
\begin{minipage}[c]{0.4\linewidth}
\centering
$\alpha=1.5~~~~~~$\\
\includegraphics[trim={30cm 0 20cm 0},clip,width=4cm,height=3.7cm,scale=1.0]{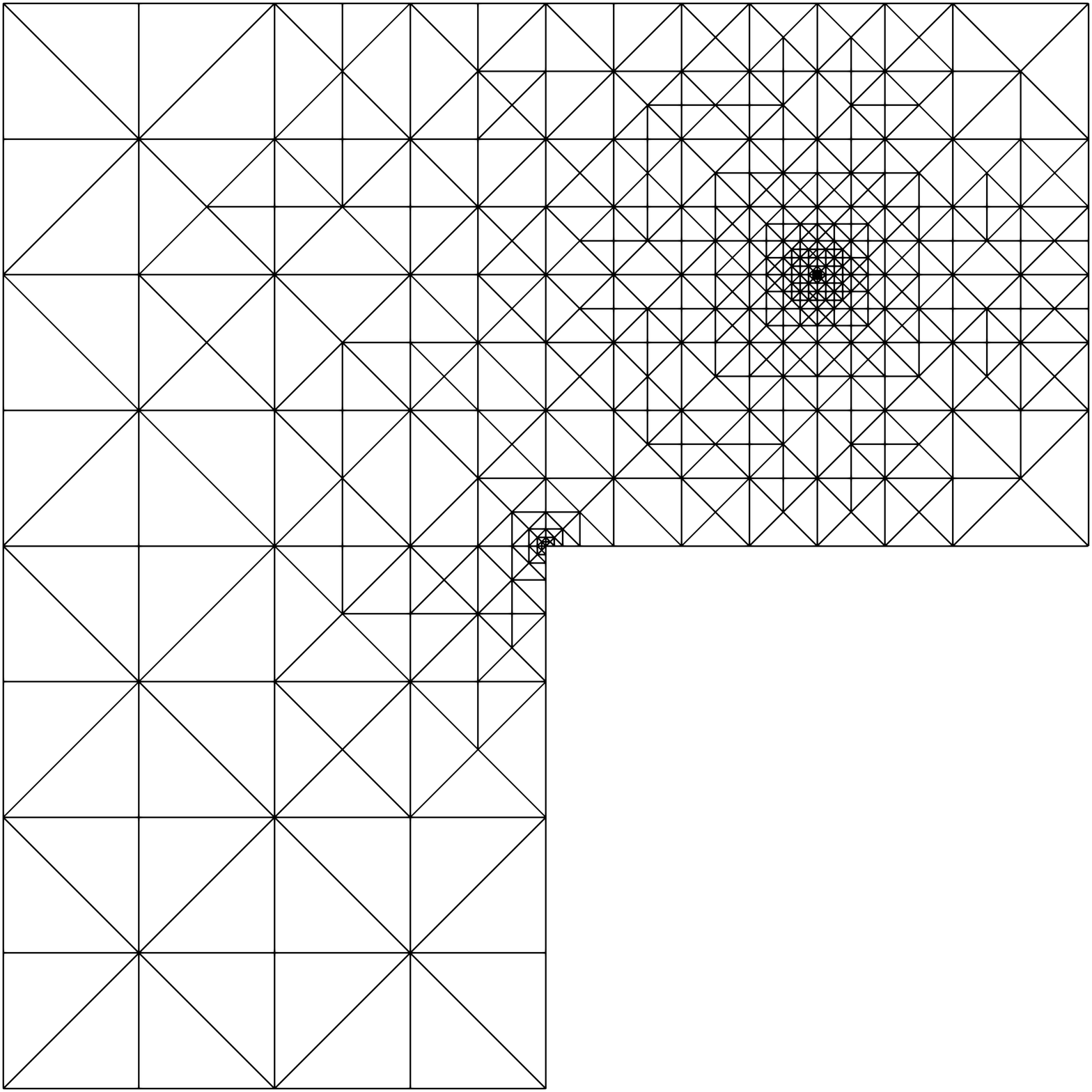}\\
$\textrm{(d)}~~~~~$
\end{minipage}
\begin{minipage}[c]{0.45\linewidth}
\centering
$\alpha=1.9~~~~~~$\\
\includegraphics[trim={30cm 0 20cm 0},clip,width=4cm,height=3.7cm,scale=1.0]{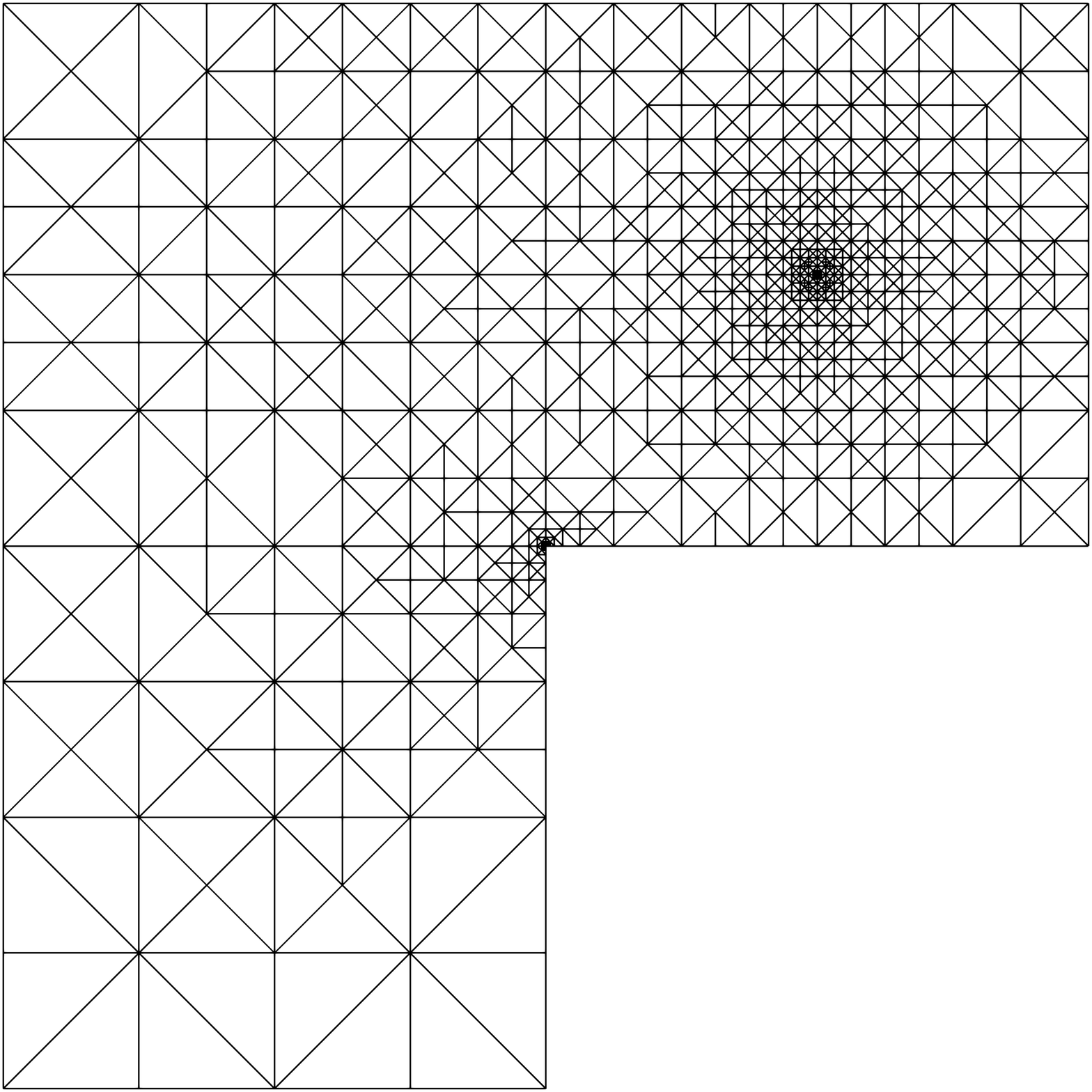}\\
(e)
\end{minipage}
\caption{Example 2: Meshes obtained after 30 iterations of our adaptive loop for (a) $\alpha =0.1$ (328 elements and 181 vertices);  (b) $\alpha = 0.5$ (328 elements and 181 vertices); (c) $\alpha=1.0$ (539 elements and 291 vertices); (d) $\alpha=1.5$ (847 elements and 450 vertices); and (e) $\alpha = 1.9$ (1380 elements and 728 vertices).}
\label{fig_4}
\end{figure}

\begin{figure}[ht]
\begin{minipage}[c]{0.3\linewidth}
\includegraphics[trim={0cm 0 0cm 0},clip,width=5cm,height=4cm,scale=1.0]{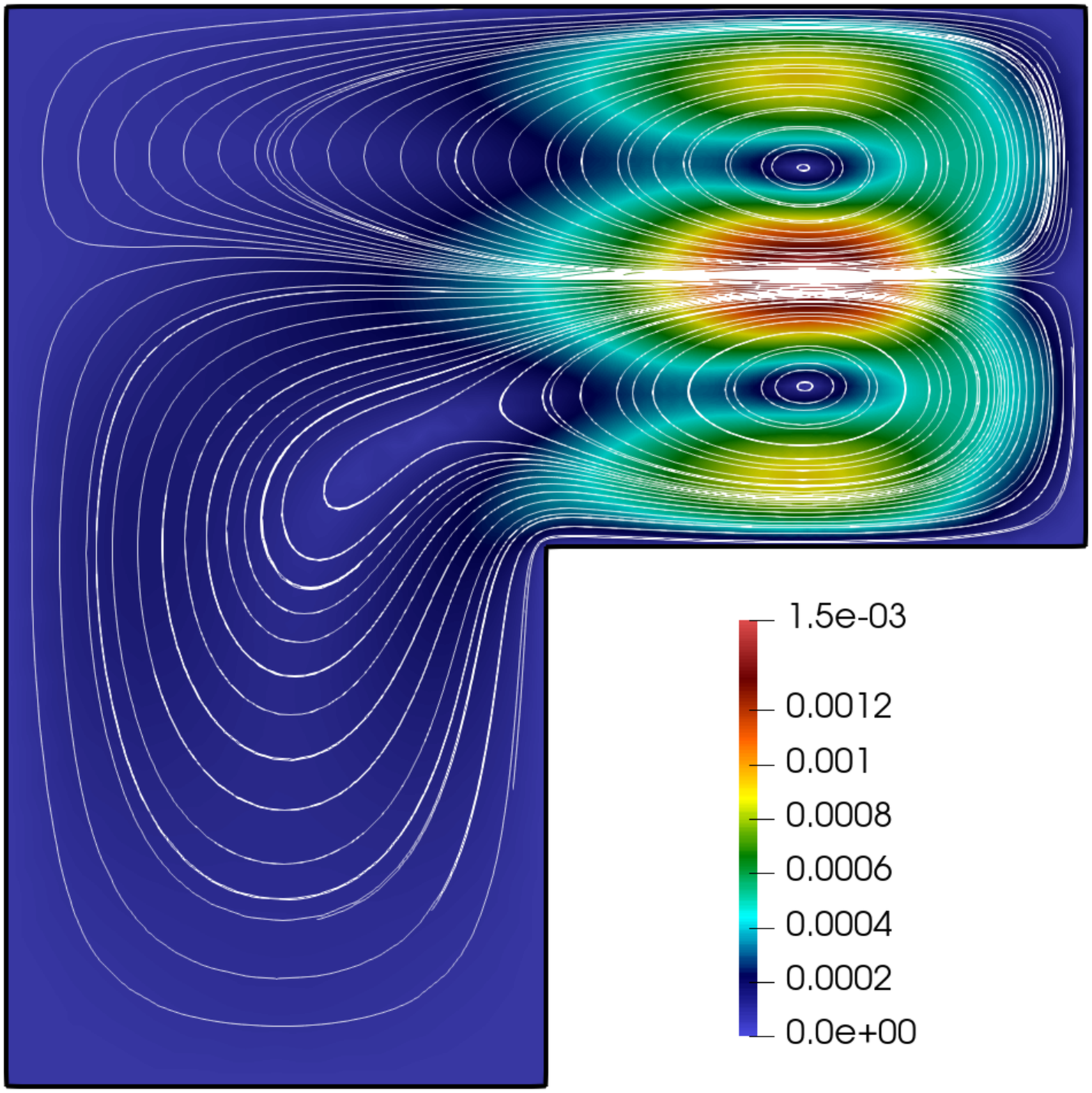}
\end{minipage}
\begin{minipage}[c]{0.3\linewidth}
\includegraphics[trim={0cm 0 0cm 0},clip,width=4.5cm,height=4cm,scale=1.0]{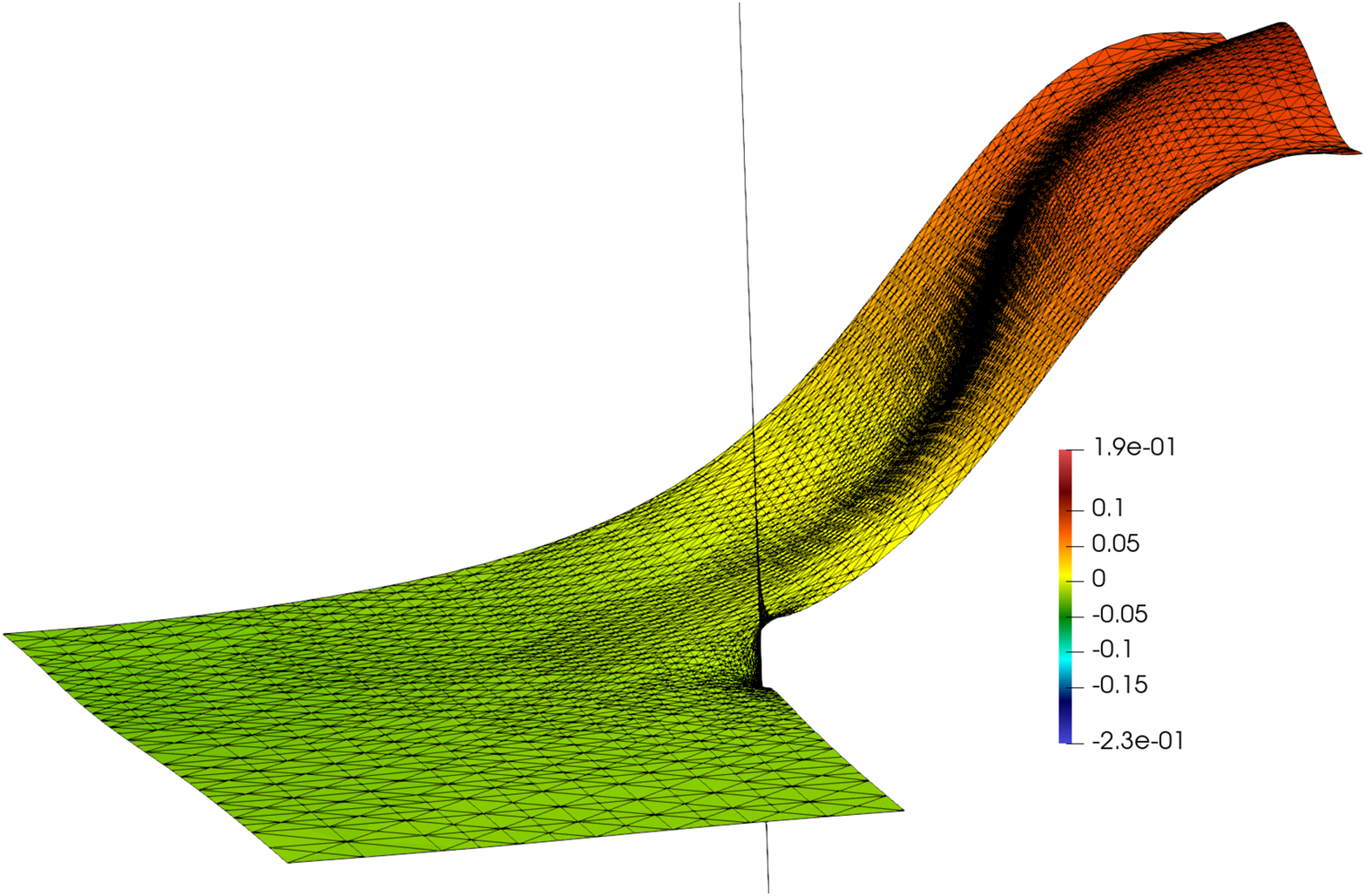}
\end{minipage}
\begin{minipage}[c]{0.34\linewidth}
\includegraphics[trim={0cm 0 0cm 0},clip,width=4.1cm,height=4cm,scale=1.0]{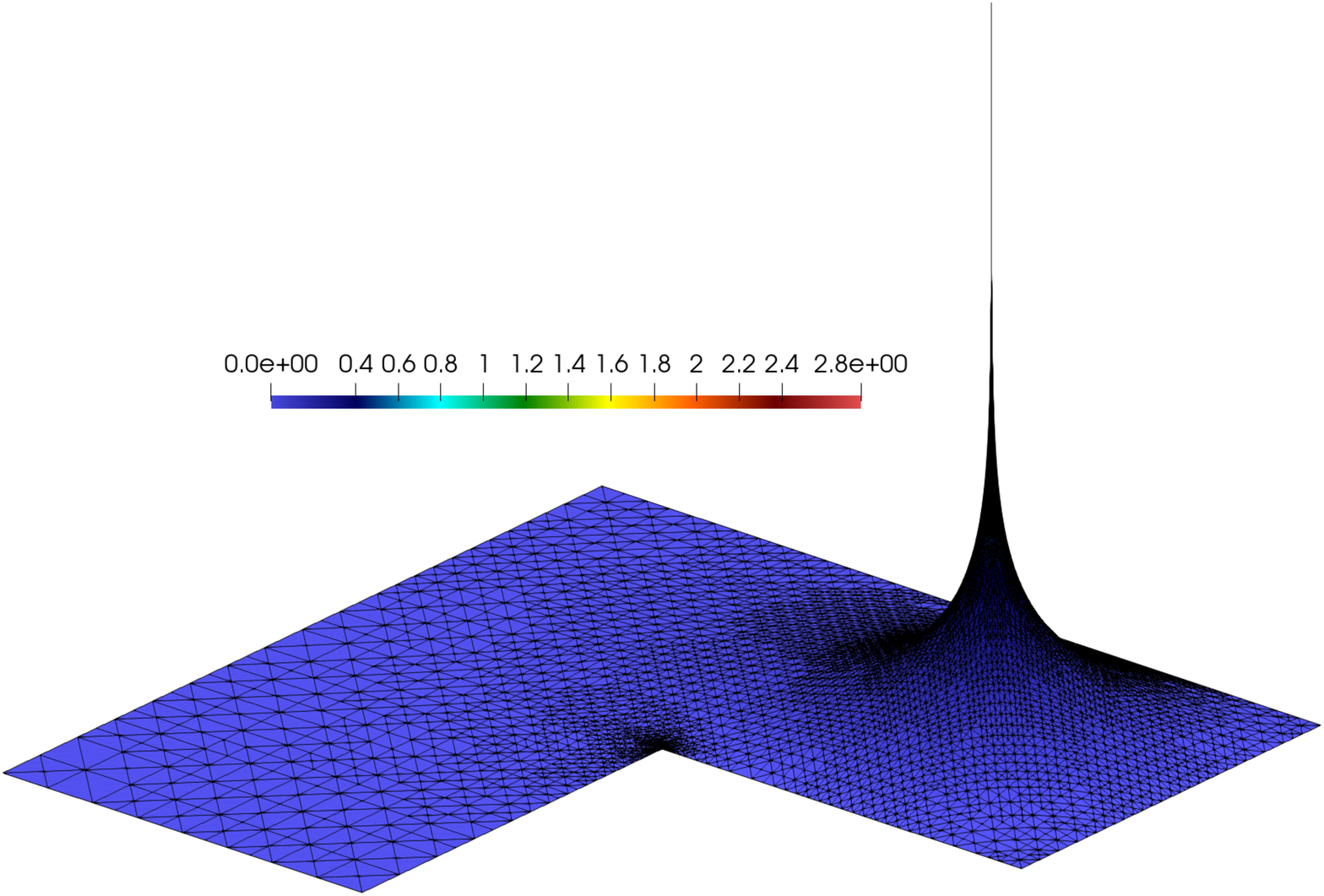}
\end{minipage}
\caption{Example 2: Finite element approximations of $|\ue_h|$, combined with its streamlines (left), pressure $\pe_h$ (center), and temperature $\Te_h$ (right) over a mesh containing 16105 elements and 8178 vertices obtained after 65 adaptive refinements ($\alpha=1.5$).}
\label{fig_5}
\end{figure}

\bibliographystyle{siamplain}
\bibliography{biblio}
\end{document}